\documentclass{elsarticle}

\usepackage{amssymb,amsmath,amscd,amsfonts,color,amsthm, bbm}
\usepackage{graphicx} 

\newcommand{\RR}{{\mathbb{R}}}

\newcommand{\CC}{{\mathbb{C}}}

\newcommand{\1}{{\mathbbm{1}}}   % Needs \usepackage{bbm}
\DeclareMathOperator*{\var}{var}
\DeclareMathOperator*{\tr}{tr}
\DeclareMathOperator*{\sinc}{sinc}
\DeclareMathOperator*{\support}{supp}
\DeclareMathOperator*{\interior}{Int}
\newcommand{\bs}{\boldsymbol}
\newcommand{\EE}{{\mathbb{E}}}
\DeclareMathOperator*{\diag}{diag}
\newcommand{\toas}{\xrightarrow{\text{a.s.}}}
\newcommand{\toP}{\xrightarrow{\mathcal{P}}}
\newcommand{\toL}{\xrightarrow{\mathcal{L}}}
\newcommand{\T}{{\mathrm{T}}}
\newtheorem{lemma}{Lemma}
\newtheorem{corollary}{Corollary}
\newtheorem{proposition}{Proposition}
\newtheorem{theorem}{Theorem}

\newtheorem{remark}{Remark}
\def\adots{
  \mathinner{\mkern1mu\raise1pt\hbox{.}\mkern2mu\raise4pt\hbox{.}
  \mkern2mu\raise7pt\vbox{\kern7pt\hbox{.}}\mkern1mu}}
\def\build#1_#2^#3{\mathrel{
\mathop{\kern 0pt#1}\limits_{#2}^{#3}}}
\newtheorem{assumption}{{\bf Assumption}}

\theoremstyle{remark}

\begin{document}

\begin{frontmatter}

\title{A Subspace Estimator for Fixed Rank Perturbations \\ 
of Large Random Matrices} 

\author[cnrs]{Walid Hachem\corref{cor1}} 
\ead{walid.hachem@telecom-paristech.fr} 

\author[umlv]{Philippe Loubaton} 
\ead{loubaton@univ-mlv.fr} 

\author[cttc]{Xavier Mestre}
\ead{xavier.mestre@cttc.cat} 

\author[cnrs]{Jamal Najim} 
\ead{jamal.najim@telecom-paristech.fr} 

\author[umlv]{Pascal Vallet} 
\ead{vallet@univ-mlv.fr}

\address[cnrs]{
CNRS ; T\'el\'ecom Paristech, 46, rue Barrault, 75013 Paris, France.} 
\address[umlv]{
IGM LabInfo, UMR 8049, Institut Gaspard Monge,
Universit\'e Paris Est Marne-la-Vall\'ee, \\ 
5, Bd Descartes, Champs sur Marne, 
77454 Marne La Vall\'ee Cedex 2, France.} 
\address[cttc]{
Centre Tecnol\`ogic de Telecomunicacions de Catalunya (CTTC), 
Parc Mediterrani de la Tecnologia - Building B4, 
Av.~Carl Friedrich Gauss 7, 
08860 - Castelldefels 
Barcelona, Spain.}

\cortext[cor1]{Corresponding author}

\date{\today} 

\begin{abstract}
  This paper deals with the problem of parameter estimation based on
  certain eigenspaces of the empirical covariance matrix of an
  observed multidimensional time series, in the case where the time
  series dimension and the observation window grow to infinity at the
  same pace.  In the area of large random matrix theory, recent
  contributions studied the behavior of the extreme eigenvalues of a
  random matrix and their associated eigenspaces when this matrix is
  subject to a fixed-rank perturbation.  The present work is concerned
  with the situation where the parameters to be estimated determine
  the eigenspace structure of a certain fixed-rank perturbation of the
  empirical covariance matrix. An estimation algorithm in the spirit
  of the well-known MUSIC algorithm for parameter estimation is
  developed. It relies on an approach recently developed by
  Benaych-Georges and Nadakuditi \cite{ben-rao-published, ben-rao-11},
  relating the eigenspaces of extreme eigenvalues of the empirical
  covariance matrix with eigenspaces of the perturbation matrix.
  First and second order analyses of the new algorithm are performed.
\end{abstract} 


\end{frontmatter}

\section{Introduction}
\label{intro}

Parameter estimation algorithms based on the estimation of an eigenspace  
of the autocorrelation matrix of an observed multivariate time series are very
popular in the areas of statistics and signal processing. 
Applications of such algorithms include the estimation of the angles of 
arrival of plane waves impinging on an array of antennas, the estimation
of the frequencies of superimposed sine waves, or the resolution of multiple 
paths of a radio signal. 
Denoting by $N$ the signal dimension (\emph{e.g.}, the number of antennas) 
and by $n$ the length of the time observation window, the observed time
series is represented by a $N \times n$ random matrix 
$\Sigma_n = X_n + P_n$ where $X_n$ and $P_n$ are respectively the 
so-called noise and signal matrices. In many applications, $P_n$ is 
represented as 
\begin{equation} 
\label{def-P-DOA} 
P_n 
% = \sum_{k=1}^r b_n(\varphi_k) s_{k,n}^* 
= B(\varphi_1, \cdots \varphi_r) S_n^* \ ,
\end{equation} 
where $(\varphi_1, \ldots, \varphi_r)$ are the $r \leq \min(N,n)$ 
deterministic parameters to be estimated, $B$ is a $N \times r$ matrix of 
the form $B(\varphi_1, \cdots \varphi_r) = 
\begin{bmatrix} b(\varphi_1) &\cdots & b(\varphi_r) \end{bmatrix}$
where $b(\varphi)$ is a known $\CC^N$-valued function of $\varphi$,
and the $S_n$ is an unknown $n \times r$ matrix with rank $r$
representing the signals transmitted by the $r$ emitting sources.  As
usual (and unless stated otherwise), $A^*$ stands for the Hermitian
adjoint of matrix $A$. It will be assumed in this work that this
matrix is deterministic. Often, the noise matrix $X_n$ is a complex random
matrix such that the real and imaginary parts of its elements are $2Nn$
independent random variables with common probability law
$\mathcal{N}(0,1/(2n))$. In this case, we shall say that
$\sqrt{n} X_n$ is a standard Gaussian matrix.

We shall consider here ``direction of arrival'' vector functions $b(\varphi)$ 
that are typically met in the field of antenna processing. 
These functions are written $$b(\varphi) = 
N^{-1/2} \begin{bmatrix} \exp( - \imath D \ell \varphi ) 
\end{bmatrix}_{\ell=0}^{N-1}$$ with domain $\varphi \in [0, \pi/D]$
where $D$ is a positive real constant and $\imath^2=-1$. Assuming that
the angular parameters $\varphi_k$ are all different, the well-known
MUSIC (MUltiple SIgnal Classification,
\cite{schmidt-MUSIC,bienvenu-MUSIC}) algorithm for estimating these
parameters from $\Sigma_n$ relies on the following simple idea: Assume
that $\sqrt{n} X_n$ is standard Gaussian and let $\Pi$ be the
orthogonal projection matrix on the eigenspace of $\EE
\Sigma_n\Sigma_n^* = B S_n^* S_n B^* + I_N$ associated with the $r$
largest eigenvalues, where $I_N$ is the $N\times N$ identity
matrix. Obviously, $\Pi$ is the orthogonal projector on the column
space of $B(\varphi_1, \ldots, \varphi_r)$. As a consequence, the
angles $\varphi_k$ coincide with the zeros of the function
$b(\varphi)^* (I - \Pi) b(\varphi)$ on $[0, \pi/D]$.  Since $\|
b(\varphi) \| = 1$, they equivalently coincide with the maximum values
(at one) of the so-called localization function $\chi(\varphi) =
b(\varphi)^* \Pi b(\varphi)$.

In practice, $\Pi$ is classically replaced with the orthogonal
projection matrix $\widehat\Pi$ on the eigenspace associated with the
$r$ largest eigenvalues of $\Sigma_n \Sigma_n^*$.  Assuming $N$ is
fixed and $n\to\infty$, and assuming furthermore that $S_n^* S_n$
converges to some matrix $O > 0$ in this asymptotic regime, the
$\Sigma\Sigma^* \toas BOB^* + I_N$ by the Law of Large Numbers
(a.s. stands for almost surely). Hence, the random variable
$\chi_{\text{classical}}(\varphi) = b(\varphi)^* \widehat\Pi
b(\varphi)$ a.s.~converges to $\chi(\varphi)$, and it is standard to
estimate the arrival angles as
local maxima of $\chi_{\text{classical}}(\varphi)$. 

However, in many practical situations, the signal dimension $N$ and
the window length $n$ are of the same order of magnitude in which case
the spectral norm of $\widehat\Pi - \Pi$ is not small, as we shall see
below. In these situations, it is often more relevant to assume that
both $N$ and $n$ converge to infinity at the same pace, while the
number of parameters $r$ is kept fixed. The subject of this paper is
to develop a new estimator better suited to this asymptotic regime,
and to study its first and second order behavior with the help of large
random matrix theory.

In large random matrix theory, much has been said about the spectral
behavior of $X_n X_n^*$ in this asymptotic regime, for a wide range of
statistical models for $X_n$. In particular, it is frequent that the
spectral measure of this matrix converge to a compactly supported
limiting probability measure $\pi$, and that the extreme eigenvalues
of $X_n X_n^*$ a.s.~converge to the edges of this support. Considering
that $\Sigma_n$ is the sum of $X_n$ and a fixed-rank perturbation, it
is well-known that $\Sigma_n \Sigma_n^*$ also has the limiting
spectral measure $\pi$ \cite[Lemma~2.2]{bai-99}.  However, the largest
eigenvalues of $\Sigma_n \Sigma_n^*$ have a special behavior: Under
some conditions, these eigenvalues leave the support of $\pi$, and in
this case, their related eigenspaces give valuable information on the
eigenspaces of $P_n$. This paper shows how the angles $\varphi_k$ can
be estimated from these eigenspaces.

The problem of the behavior of the extreme eigenvalues of large random
matrices subjected to additive or multiplicative low rank
perturbations (often called ``spiked models'') have received a great
deal of interest in the recent years. In this regard, the authors of
\cite{bbp05,bk-sil06,paul-07} study the behavior of the extreme
eigenvalues of a sample covariance matrix when the population
covariance matrix has all but finitely many eigenvalues equal to one,
a problem described in \cite{joh01}.  Reference \cite{cdf09} is
devoted to the extreme eigenvalues of a Wigner matrix that incurs a
fixed-rank additive perturbation.  Fluctuations of these eigenvalues
are studied in
\cite{bbp05,peche-06,paul-07,by08,cdf09,cdf-clt09,bgm-fluct10}. 

Recently, Benaych-Georges and Nadakuditi proposed in
\cite{ben-rao-published, ben-rao-11} a powerful technique for
characterizing the behavior of extreme eigenvalues and their
associated eigenspaces for three generic spiked models: The models
$X_n+ P_n$ and $(I_n + P_n) X_n$ when both $X_n$ and $P_n$ are
Hermitian and $P_n$ is low-rank, and the model that encompasses ours
$(X_n + P_n)(X_n+P_n)^*$ where $X_n$ and $P_n$ are rectangular.  One
feature of this approach is that it uncovers simple relations between
the extreme eigenvalues and their associated eigenspaces on the one
hand, and certain quadratic forms involving resolvents related with the
non-perturbed matrix $X_n$ on the other. This makes the method
particularly well-suited (but not limited to) the situation where
$X_n$ is unitarily or bi-unitarily invariant, a situation that we
shall consider in this paper.  Indeed, in this situation, these
quadratic forms exhibit a particularly simple behavior in the
considered large dimensional asymptotic regime.

In this paper, we make use of the approach of 
\cite{ben-rao-published, ben-rao-11}
to develop a new subspace estimator of the angles $\varphi_k$ based on 
the eigenspaces of the isolated eigenvalues of $\Sigma_n\Sigma_n^*$. 
We perform the first and second order analyses of this estimator that we
call the ``Spike MUSIC'' estimator. 
Our mathematical developments differ somehow from those of 
\cite{ben-rao-published, ben-rao-11} and could have their own interest. 
They are based on two simple ingredients: The first is an analogue 
of the Poincar\'e-Nash inequality for the Haar distributed unitary matrices 
which has been recently discovered by Pastur and Vasilchuk 
\cite{pas-vas-book07}, and the second is a contour integration method 
by means of which the first and second order analyses are done. 
The key step of the second order analysis of our estimator lies in 
the establishment of a Central Limit Theorem on the quadratic forms 
$b(\varphi_i)^* \widehat\Pi_i b(\varphi_i)$ where the $\widehat\Pi_i$ 
are the orthogonal projection matrices on certain eigenspaces of 
$\Sigma_n\Sigma_n^*$ associated with the isolated eigenvalues. The employed 
technique can easily be used to study the fluctuations of projections of other
types of vectors on these eigenspaces. 
 
% Le spiked population model etudie par \cite{bbp05} et \cite{bk-sil06} a 
% ete pressenti par Johnstone dans \cite{joh01}. 

% \cite{bbp05}: Wishart gaussien. Population covariance matrix $=I+P$. 

% \cite{bk-sil06}: Wishart non gaussien. Population covariance matrix $=I+P$. 

% \cite{cdf09}: Wigner with additive perturbation. 

% \cite{by08} : fluctuations des spikes de \cite{bk-sil06}. 

% \cite{cdf-clt09} : fluctuations du cas Wigner deformation additive. 

% \cite{bgm-fluct10} : flucutations pour leur modele general. 

We now state our general assumptions and introduce some notations. 

\subsection*{Assumptions and Notations} 
We now state the general assumptions of the paper. Consider the
sequence of $N \times n$ matrices $\Sigma_n = X_n +P_n$ where:
\begin{assumption}\label{ass:asymptotics}
The dimensions $N,n$ satisfy: $N \leq n$, $n \to \infty$ and 
$$
\frac Nn \to c \in (0, 1] 
$$ (notation
for this asymptotic regime: $n\to\infty$). 
\end{assumption}

The following assumption on $X_n$ is widely used in the random matrix
literature \cite{hia-pet-book00,pas-shc-book11}: 

\begin{assumption}\label{ass:X}
Matrices $X_n$ are random $N \times n$ bi-unitarily invariant matrices, 
\emph{i.e.}, each $X_n$ admits the singular value decomposition 
$X_n = L_n \Gamma_n R_n^*$ where $L_n$, the $N\times N$ matrix $\Gamma_n$ 
and $R_n$ are independent,  
$L_n$ is Haar distributed on the group ${\mathcal U}(N)$ of unitary $N\times N$
matrices, and $R_n$ is a $n \times N$ submatrix of a Haar distributed matrix
on ${\mathcal U}(n)$. 
\end{assumption}

% \begin{itemize}
% \item[{\bf A1}] 
% The dimensions $N,n$ satisfy: $N \leq n$, $n \to \infty$ and 
% $$
% \frac Nn \to c \in (0, 1] 
% $$ (notation
% for this asymptotic regime: $n\to\infty$). 
% \end{itemize}
% The assumptions on $X_n$ are the following:
% \begin{itemize} 
% \item[{\bf A2}] 
% Matrices $X_n$ are random $N \times n$ bi-unitarily invariant matrices, 
% \emph{i.e.}, each $X_n$ admits the singular value decomposition 
% $X_n = L_n \Gamma_n R_n^*$ where $L_n$, the $N\times N$ matrix $\Gamma_n$ 
% and $R_n$ are independent,  
% $L_n$ is Haar distributed on the group ${\mathcal U}(N)$ of unitary $N\times N$
% matrices, and $R_n$ is a $n \times N$ submatrix of a Haar distributed matrix
% on ${\mathcal U}(n)$. 
% \end{itemize}

We recall that the Stieltjes transform of a probability measure $\pi$ on the
real line is the complex function 
\[
m(z) = \int \frac{1}{t-z} \pi(dt) \ ,
\]
analytic on $\CC_+ = \{ z : \Im(z) > 0 \}$. 
\begin{assumption}\label{ass:resolvent}
Let $Q_n(z) = (X_n X_n^* - zI_N)^{-1}$ be the resolvent associated with 
$X_nX_n^*$ and let $\alpha_n(z) = N^{-1} \tr Q_n(z)$. For every $z \in \CC_+$, 
$\alpha_n(z)$ 
a.s.~converges to a deterministic function $m(z)$ which is the 
Stieltjes transform of a probability measure $\pi$ supported by the 
compact interval $[\lambda_-, \lambda_+]$.
\end{assumption}

\begin{assumption}\label{ass:eigen-max}
The quantity $\| X_nX_n^* \|$ a.s.~converges to $\lambda_+$ as $n\to\infty$, where 
$\| \cdot \|$ denotes the spectral norm. 
\end{assumption}

% \begin{itemize}
% \item[{\bf A3}] 
% Let $Q_n(z) = (X_n X_n^* - zI_N)^{-1}$ be the resolvent associated with 
% $X_nX_n^*$ and let $\alpha_n(z) = N^{-1} \tr Q_n(z)$. For every $z \in \CC_+$, 
% $\alpha_n(z)$ 
% a.s. converges to a deterministic function $m(z)$ which is the 
% Stieltjes transform of a probability measure $\pi$ supported by the 
% compact interval $[\lambda_-, \lambda_+]$.
% \item[{\bf A4}] 
% The quantity $\| X_nX_n^* \|$ a.s. converges to $\lambda_+$ as $n\to\infty$, where 
% $\| \cdot \|$ denotes the spectral norm. 
% \end{itemize}

Let $\widetilde Q_n(z) = (X_n^* X_n - zI_n)^{-1}$ and
$\tilde\alpha_n(z) = n^{-1} \tr\widetilde Q_n(z)$. Equivalently to the
convergence assumed by Assumption \ref{ass:resolvent}, one may assume that
$\tilde\alpha_n(z)$ a.s.~converges on $\CC_+$ to a
deterministic function $\tilde m(z)$ which is the Stieltjes transform
of a probability measure $\tilde\pi$. In that case, $\tilde m(z) = c
m(z) - (1-c)/z$ and $\tilde\pi = c \pi + (1-c)\delta_0$.

\begin{remark}
\label{rem-gauss} 
In the areas of signal processing and communication theory, the 
noise matrix $X_n$ satisfying Assumptions {\bf A2-A4} is such that 
$\sqrt{n} X_n$ is standard Gaussian - see for instance 
\cite{marchenko-pastur}, \cite{geman-80}. 
\end{remark} 

We first make a general assumption on matrices $P_n$; it will be specified 
later, and adapted to the context of the MUSIC algorithm: 
\begin{assumption}\label{ass:P-rough}
Matrices $P_n$ are deterministic with a fixed rank equal to $r$ for all 
$n$ large enough. Denoting by $P_n = U_n \Omega_n V_n^*$ a singular value 
decomposition of $P_n$, the matrix of singular values 
$\Omega_n = \diag(\omega_{1,n}, \ldots, \omega_{r,n})$ 
with $\omega_{1,n} \geq \omega_{2,n} \geq \cdots \geq \omega_{r,n}$ 
converges to  
\begin{equation} 
\label{eq-limit-O} 
O = \begin{bmatrix} 
\omega_1 I_{j_1} & & \\ & \ddots & \\
& & \omega_{s} I_{j_s} \end{bmatrix}  \ ,
\end{equation} 
where $\omega_1 > \cdots > \omega_s > 0$ and 
$j_1 + \cdots + j_s = r$. 
\end{assumption}

% \begin{itemize} 
% \item[{\bf A5}] 
% Matrices $P_n$ are deterministic with a fixed rank equal to $r$ for all 
% $n$ large enough. Denoting by $P_n = U_n \Omega_n V_n^*$ a singular value 
% decomposition of $P_n$, the matrix of singular values 
% $\Omega_n = \diag(\omega_{1,n}, \ldots, \omega_{r,n})$ 
% with $\omega_{1,n} \geq \omega_{2,n} \geq \cdots \geq \omega_{r,n}$ 
% converges to  
% \begin{equation} 
% \label{eq-limit-O} 
% O = \begin{bmatrix} 
% \omega_1 I_{j_1} & & \\ & \ddots & \\
% & & \omega_{s} I_{j_s} \end{bmatrix}  \ ,
% \end{equation} 
% where $\omega_1 > \cdots > \omega_s > 0$ and 
% $j_1 + \cdots + j_s = r$. 
% \end{itemize}

\subsubsection*{Notations.}
As usual, if $z\in \mathbb{C}$, we shall denote by $\Re(z)$ and
$\Im(z)$ its real and imaginary parts. We shall denote by
$\xrightarrow[]{a.s.}$ (resp.  $\xrightarrow[]{\mathcal P}$,
$\xrightarrow[]{\mathcal D}$) the almost sure convergence
(resp. convergence in probability, in distribution).
We denote by $\delta_{i,j}$ the Kronecker delta ($=1$ if $i=j$ and
$0$ otherwise). 
 
The eigenvalues of $\Sigma_n\Sigma_n^*$ are $\hat\lambda_{1,n} \geq
\hat\lambda_{2,n} \geq \cdots \geq \hat\lambda_{N,n}$. Associated
eigenvectors will be denoted $\hat u_{1,n}, \hat u_{2,n}, \cdots, \hat
u_{N,n}$.  For $k \in \{1,\ldots,r\}$, we shall denote by $i(k)$ the
index $i\in\{1,\ldots,s\}$ such that $j_1+\cdots+j_{i-1} < k \leq
j_1+\cdots+j_i$.  For $i=1,\ldots,s$, We shall denote by
$\widehat\Pi_{i,n}$ the orthogonal projection matrix on the eigenspace
of $\Sigma_n \Sigma_n^*$ associated with the eigenvalues
$\hat\lambda_{k,n}$ such that $i(k) = i$, \emph{i.e.},
$\widehat\Pi_{i,n} = \sum_{k:i(k) = i} \hat u_{k,n} \hat u_{k,n}^*$
when this eigenspace is defined.  Columns of $U_n$ (see {\bf A5}) will
be denoted $u_{1,n}, \cdots, u_{r,n}$.  Given $i$, the orthogonal
projection matrix on the eigenspace of $P_n P_n^*$ associated with the
eigenvalues $\omega_{k,n}^2$ such that $i(k) = i$ will be $\Pi_{i,n} =
\sum_{k:i(k) = i} u_{k,n} u_{k,n}^*$. Indexes $n$ and $N$ will often
be dropped for readability.

\subsection*{Paper organization}
The paper is organized as follows. Section \ref{prelim} is devoted to the
mathematical preliminaries. The general approach is described in Section 
\ref{vap-spikes}. The Spike MUSIC algorithm is presented in Section 
\ref{spk-music} along with a first order study of this algorithm. 
Fluctuations of the estimates of the $\varphi_k$ are studied in Section
\ref{sec-clt} under the form of a Central Limit Theorem.

\section{Preliminary mathematical results} 
\label{prelim} 

We shall need the two following results. The first one is well-known
\cite{pas-vas-book07}.  The second result, due to Pastur and
Vasilchuk, is the unitary analogue of the well-known Poincar\'e-Nash
inequality.
\begin{lemma}
\label{lm-moments-Haar}
Let $W = [w_{ij} ]$ be a random matrix Haar distributed on ${\mathcal U}(n)$. 
Then 
\[
\EE\left[ w_{ij} w^*_{i'j'} \right] = \frac 1n \delta_{i,i'} \delta_{j,j'}\ .
\]
\end{lemma} 
\begin{lemma}[\cite{pas-vas-book07,pas-shc-book11}] 
\label{lm-NP-like}
Let $\Phi : {\mathcal U}(n) \to \CC$ be a function that admits a $C^1$ 
continuation to an open neighborhood of ${\mathcal U}(n)$ in the whole algebra
of $n \times n$ complex matrices. Then
\[
\var \Phi(W_n) = \EE \left| \Phi(W_n) \right|^2 - 
\left| \EE \Phi(W_n) \right|^2 \leq 
\frac 1n \sum_{j,k=1}^n 
\EE 
\left| \Phi'(W_n) \cdot \left( {\bf e}_j {\bf e}_k^\T W_n \right) \right|^2
\] 
where $\EE$ is the expectation with respect to the Haar measure on 
${\mathcal U}(n)$, where $\Phi'$ is the differential of $\Phi$ as a 
function on $\RR^{2n^2}$ acting on the matrix ${\bf e}_j {\bf e}_k^\T W_n$
seen as an element of $\RR^{2n^2}$, and where 
${\bf e}_j = [0 \cdots 0 \, 1 \, 0 \cdots 0]^*$ is the $j^{\text{th}}$ 
canonical vector of $\CC^n$. 
\end{lemma}

Given a small $\varepsilon_1 > 0$, let $O_n$ be the probability
event 
\begin{equation}\label{eq:def-O}
O_n = \left\{ \| X_nX_n^* \| \leq \lambda_+ + \varepsilon_1 
\right\}\ .
\end{equation}
By Assumption \ref{ass:eigen-max}, $\1_{O_n} \toas 1$ as $n\to \infty$.

\begin{lemma}
\label{lm-fq-haar} 
Let Assumption \ref{ass:X} holds true and let $u, v$ be two unit norm deterministic $N \times 1$ 
vectors such that $u^* v = 0$. Then for any $z$ with
$\Re(z) > \lambda_+ + \varepsilon_1$, 
\begin{eqnarray*} 
\EE\left| \1_{O_n} \times 
u^* \left( Q(z) - \alpha(z) I \right) u \right|^p &\leq& 
\frac{K_p}{N^{p/2} d(z, \lambda_+ + \varepsilon_1)^p} \ ,\\
\EE\left| \1_{ O_n} \times  u^* Q(z) v \right|^p &\leq& 
\frac{K_p}{N^{p/2} d(z, \lambda_+ + \varepsilon_1)^p} \ ,
\end{eqnarray*} 
where the constant $K_p$ only depends on $p$, and where $d(z,z')$ is
the Euclidean distance between $z$ and $z'$ in $\mathbb {C}$.
\end{lemma} 
\begin{proof} Recall that $X=L\,\Gamma\, R^*$ by Assumption \ref{ass:X};
  let $D = (\Gamma^2 -zI)^{-1}$; write:
\[
\begin{bmatrix} & u^* & \\ & v^* & \end{bmatrix} 
\left( Q - \alpha I \right) 
\begin{bmatrix} \\ u & v \\ \\ \end{bmatrix} 
= 
\begin{bmatrix} & w_1^* & \\ & w_2^* & \end{bmatrix} 
\left( D  - \frac{\tr D}{N}I \right) 
\begin{bmatrix} \\ w_1 & w_2 \\ \\ \end{bmatrix}  . 
\]
Thanks to {\bf A2}, $w_1$ and $w_2$ are the first two columns of a 
$N \times N$ unitary Haar distributed matrix $W = [w_{ij} ]$ independent of 
$D$. Let $M = \1_{O_n} 
\times \left( D - N^{-1} (\tr D) I  \right)$ and 
$\Phi_i(W) =  w_1^* M w_i$ for $i=1,2$. 
Then $\EE \Phi_1(W) = \EE \Phi_2(W) = 0$ by Lemma \ref{lm-moments-Haar}. 
Applying Lemma \ref{lm-NP-like} to $\Phi_i$ after noticing that 
$\Phi_i'(W) \cdot A = {\bf e}_1^\T A^* M w_i + w_1^*  M A {\bf e}_i$ for
any $N\times N$ matrix $A$, we obtain:
\begin{align*}
\EE| \Phi_i |^2  = \var(\Phi_i) &\leq 
\frac 1N \sum_{j,k=1}^N \EE\left| w_{k1}^* [ MW ]_{ji} + [W^* M]_{1j} w_{ki} 
\right|^2 \ ,\\
&\leq \frac 2N \EE \left( \| M w_i \|^2 +  \| M w_1 \|^2 \right) \ ,\\
&\leq \frac{8}{N d(z, \lambda_+ +\varepsilon_1)^2} \ . 
\end{align*} 
We now proceed by induction; assume that the result is true until
$p\geq 1$. Applying Lemma \ref{lm-NP-like} to $\Phi_i^{(p+1)/2}$, we
obtain:
\begin{align*}
\var\Bigl(\Phi_i^{\frac{p+1}{2}}\Bigr) &\leq 
\frac 1N \sum_{j,k=1}^N 
\EE\left| \frac{p+1}{2} \Phi_i^{\frac{p-1}{2}} \Phi'_i(W) \cdot
\left( {\bf e}_j {\bf e}_k^\T W \right) \right|^2 \ ,\\
&\leq 
\frac{(p+1)^2}{2N} \EE \left( \left| \Phi_i \right|^{p-1}
\left( \| M w_i \|^2 +  \| M w_1 \|^2 \right) \right) \ ,\\
&\leq
\frac{2 (p+1)^2 K_{p-1}}{d(z, \lambda_+ +\varepsilon_1)^{p+1} N^{(p+1)/2}}\ . 
\end{align*} 
Using again the induction hypothesis, we get: 
\begin{multline*} 
\EE\left| \Phi_i \right|^{p+1} = 
\var\Bigl(\Phi_i^{\frac{p+1}{2}}\Bigr) + 
\left| \EE \Phi_i^{\frac{p+1}{2}} \right|^2 \\ 
\leq
\frac{2 (p+1)^2 K_{p-1} + K_{(p+1)/2}^2} 
{d(z, \lambda_+ +\varepsilon_1)^{p+1}  N^{(p+1)/2}} 
= \frac{K_{p+1}}{d(z, \lambda_+ +\varepsilon_1)^{p+1}  N^{(p+1)/2}}\ ,
\end{multline*} 
which concludes the proof.
\end{proof} 

\begin{lemma}
\label{lm-fq-two-haars} 
Let Assumption \ref{ass:X} hold true; let $u, v$ be two unit norm
deterministic vectors with respective dimensions $N \times 1$ and $n
\times 1$.  Then for any $z$ such as $\Re(z) > \lambda_+ +
\varepsilon_1$,
\[ 
\EE\left| \1_{O_n} \times u^* X \widetilde Q(z) v \right|^p \leq 
\frac{K_p}{n^{p/2} d(z, \lambda_+ + \varepsilon_1)^p} . 
\] 
\end{lemma} 
\begin{proof}
Let $C = \Gamma(\Gamma^2 -zI)^{-1}$. By Assumption \ref{ass:X},
$u^* X \widetilde Q(z) v = w^* C \tilde w = \Phi(w)$ where $w$ is a vector 
uniformly distributed on the unit 
sphere of $\CC^N$, $\tilde w$ is a vector uniformly distributed on the unit 
sphere of $\CC^n$ and truncated to its first $N$ elements, and 
$w$, $\tilde w$ and $C$ are independent. The lemma is proved as above
by applying Lemma \ref{lm-NP-like} to $\Phi$ and by taking the 
expectation with respect to the law of $w$. 
\end{proof} 

\begin{lemma}
\label{lm-uniform-cvg} 
Let Assumptions \ref{ass:asymptotics}-\ref{ass:eigen-max} hold true.
Let ${\mathcal C}$ be a closed path of $\CC$ such that 
$\min_{z \in {\mathcal C}} \Re(z) > \lambda_+$. Fix the integer $r \leq N$ 
and let $U_n$ and $V_n$ be two deterministic isometry matrices with dimensions
$N \times r$ and $n \times r$ respectively. Then 
\begin{eqnarray*}
\sup_{z \in {\mathcal C}} 
\| U_n^* \left( Q_n(z) - m(z) I_N \right) U_n \| 
&\xrightarrow[n\to\infty]{\text{a.s.}}& 0 \ , \\ 
\sup_{z \in {\mathcal C}} 
\| V_n^* \left( \widetilde Q_n(z) - \tilde m(z) I_n\right) V_n \| 
&\xrightarrow[n\to\infty]{\text{a.s.}}& 0  \ ,\\ 
\sup_{z \in {\mathcal C}} 
\| U_n^* X_n \widetilde Q_n(z) V_n \| 
&\xrightarrow[n\to\infty]{\text{a.s.}}& 0\ . 
\end{eqnarray*} 
\end{lemma}  
\begin{proof}
  Recall the definition \eqref{eq:def-O} of the set $O_n$ and assume
  that $\varepsilon_1$ is chosen
  such that $\min_{z \in {\mathcal C}} \Re(z) > \lambda_+ +
  \varepsilon_1$; let
\[
h_n(z) = \1_{O_n} \times U_n^* 
\left( Q_n(z) - \alpha_n(z) I_N\right) U_n\ .
\]
For any $\ell,s \leq r$, $[h_n]_{\ell,s}$ is a holomorphic function on
$\CC - [0, \lambda_+ + \varepsilon_1]$. Consider a denumerable
sequence of points $(z_k)$ in $\CC - [0, \lambda_+ + \varepsilon_1]$
with an accumulation point in that set. By Lemma \ref{lm-fq-haar} with
$p=3$, Markov inequality and Borel-Cantelli's lemma, there exists a
probability one set on which $[h_n(z_k)]_{\ell,s} \to 0$ for every
$k$. Moreover, the $\left| [h_n(z_k)]_{\ell,s} \right|$ are uniformly
bounded on any compact set of $\CC - [0, \lambda_+ + \varepsilon_1]$.
By the normal family theorem, every $n$-sequence of $[h_n]_{\ell,s}$
contains a further subsequence which converges uniformly on the
compact set ${\mathcal C} \subset \CC - [0, \lambda_+ + \varepsilon_1]$ to
a holomorphic function that we denote $h^*$. Since $h^*(z_k) = 0$ for
all $k$, $h^*(z) = 0$ on ${\mathcal C}$, hence $\left| [h_n(z)]_{\ell,s}
\right|$ converges uniformly to zero on ${\mathcal C}$ with probability
one, and thanks to Assumption \ref{ass:eigen-max}, $\| U^* \left( Q(z)
  - \alpha(z) I \right) U \| \to 0$ uniformly on ${\mathcal C}$ with
probability one.  The same argument, used in conjunction with
Assumption \ref{ass:resolvent}, shows that with probability one,
$\alpha(z) - m(z) \to 0$ uniformly on ${\mathcal C}$, and the first
assertion is proven.  The second and third assertions are proven
similarly, the third being obtained with the help of Lemma
\ref{lm-fq-two-haars}.
\end{proof} 

\section{Fixed Rank Perturbations: First Order Behavior}
\label{vap-spikes}

We first recall a result on matrix analysis that can be found in 
\cite[Th. 7.3.7]{HorJoh94}:
\begin{lemma}
\label{lm-gram-hermitian}
Given a $N \times n$ matrix $A$ with $N\le n$, let ${\bf A}$ be the matrix:
\[
{\bf A} = \begin{bmatrix} 0 & A \\ A^* & 0 \end{bmatrix}  . 
\] 
Then 
$\sigma_1, \cdots, \sigma_N$ are the singular values of $A$ if and only if
$\sigma_1, \cdots, \sigma_N, - \sigma_1, \cdots, - \sigma_N$ in addition
to $n-N$ zeros are the eigenvalues of ${\bf A}$. Furthermore, a pair
$(u, v)$ of unit norm vectors is a pair of (left,right) singular vectors
of $A$ associated with the singular value $\sigma$ if and only if 
$\begin{bmatrix} u/\sqrt{2} \\ v / \sqrt{2} \end{bmatrix}$ is a unit norm
eigenvector of ${\bf A}$ associated with the eigenvalue $\sigma$. 
\end{lemma}

Along the ideas in \cite{ben-rao-published,ben-rao-11}, we now
characterize the behavior of the largest eigenvalues of
$\Sigma\Sigma^*$, and then focus on their eigenspaces.

\subsection*{Asymptotic behavior of the largest eigenvalues of 
$\Sigma\Sigma^*$}

We start with an informal description of the approach. 
By Lemma \ref{lm-gram-hermitian}, $\lambda$ is an eigenvalue of 
$\Sigma\Sigma^*$ if and only if $\det( \bs\Sigma - \sqrt{\lambda}I) = 0$ where 
${\bs\Sigma} = \begin{bmatrix} 0 & \Sigma \\ \Sigma^* & 0 \end{bmatrix}$. 
Writing: 
\begin{equation}
\label{eq-structure-Sigma} 
{\bs\Sigma} = \begin{bmatrix} 0 & X \\ X^* & 0 \end{bmatrix} 
+ \begin{bmatrix} U & 0 \\ 0 & V \Omega \end{bmatrix}  
\begin{bmatrix} 0 & I_r \\ I_r & 0 \end{bmatrix}  
\begin{bmatrix} U^* & 0 \\ 0 & \Omega V^* \end{bmatrix}  
\stackrel{\triangle}=
{\bf X} + B J B^* \ ,
\end{equation}
and assuming that $x > 0$ is not a singular value of $X$, we have:
\begin{multline*} 
\det(\bs\Sigma - {x}I) = \det( {\bf X} - {x}I + BJB^* ) = 
\det(J) \det( {\bf X} - {x}I ) 
\det( J + B^* ( {\bf X} - {x}I )^{-1} B ) \ ,
\end{multline*}
after noticing that $J = J^{-1}$. 
Using the formula for the inversion of a partitioned matrix 
(see \cite{HorJoh94}) 
\[
\begin{bmatrix}
A_{11} & A_{12} \\ A_{12}^* & A_{22} 
\end{bmatrix}^{-1}  
= 
\begin{bmatrix}
(A_{11} - A_{12} A_{22}^{-1} A_{12}^*)^{-1} & 
- A_{11}^{-1} A_{12} ( A_{22} - A_{12}^* A_{11}^{-1} A_{12} )^{-1} \\
- ( A_{22} - A_{12}^* A_{11}^{-1} A_{12} )^{-1} A_{12}^* A_{11}^{-1} & 
(A_{22} - A_{12}^* A_{11}^{-1} A_{12} )^{-1} 
\end{bmatrix} \ ,
\]
we obtain:  
\begin{equation}
\label{eq-resolvent} 
{\bf Q}({x}) = ({\bf X} - {x}I)^{-1} = 
\begin{bmatrix} -{x}I & X \\ X^* & -{x}I \end{bmatrix}^{-1}  
= 
\begin{bmatrix}
{x} Q(x^2) & X \widetilde Q(x^2) \\ 
\widetilde Q(x^2) X^* & {x} \widetilde Q(x^2)  
\end{bmatrix} \ .
\end{equation} 
Therefore,  
\[
\det(\bs\Sigma - {x}I) = \det(J) \det( {\bf X} - {x}I ) \det \widehat H(x) \ ,
\]
where 
\[
\widehat H_n(x) = 
\begin{bmatrix}
{x} U^* Q(x^2) U & I_r + U^* X \widetilde Q(x^2) V \Omega \\
I_r + \Omega V^* \widetilde Q(x^2) X^* U & 
{x} \Omega V^* \widetilde Q(x^2) V \Omega 
\end{bmatrix}  
\]
whence for $n$ large enough, the isolated eigenvalues of
$\Sigma\Sigma^*$ above $\lambda_+$ will coincide with the zeros of
$\det \widehat H(\sqrt{x})$ that lie above $\lambda_+$.  Under
Assumptions \ref{ass:asymptotics}-\ref{ass:P-rough}, Lemma
\ref{lm-uniform-cvg} shows that $\widehat H(x)$ a.s.~converges to
\[
H(x) = \begin{bmatrix}
{x} m(x^2) I_r & I_r \\
I_r & {x} \tilde m(x^2) O^2 
\end{bmatrix} . 
\] 
Consider the equation 
\[
\det H(\sqrt{x}) = 
\det\left( x m(x) \tilde m(x) O^2 - I_r \right) = 0 \ ,
\]
and notice that the function 
\begin{equation}\label{eq:def-g}
g(x) = x m(x) \tilde m(x) = 
x 
\left( \int \frac{1}{t-x} \pi(dt) \right) 
\left( c \int \frac{1}{t-x} \pi(dt)  - \frac{1-c}{x} \right) 
\end{equation}
decreases from $g(\lambda_+^+) = \lim_{x\downarrow\lambda_+} g(x)$ to
zero on $(\lambda_+, \infty)$. Let $\omega_1^2 > \cdots > \omega_q^2$
be those among the diagonal elements of $O^2$ that satisfy $\omega_i^2
> 1/g(\lambda_+^+)$.  Equation $g(x) = \omega_i^{-2}$ will have a
unique solution $x = \rho_i > \lambda_+$ for any $i = 1, \cdots, q$,
while it will have no solution larger than $\lambda_+$ for $i > q$.
It is then expected that any eigenvalue $\hat\lambda_{k,n}$ of
$\Sigma_n\Sigma_n^*$ for which $i(k) \leq q$ (remember the
  definition of $i(k)$ provided in the paragraph ``Assumptions and
  Notations'' in Section \ref{intro}), will converge to $\rho_i$,
  while
  $\hat\lambda_{j_1+\cdots+j_{q}+1,n} \to \lambda_+$ almost surely.

  These facts are formalized in the following theorem, shown in
  \cite{ben-rao-09, ben-rao-11}:
\begin{theorem}
\label{cvg-spk}
Let Assumptions \ref{ass:asymptotics}-\ref{ass:P-rough} hold true; let
$q$ be the maximum index such that $\omega_q^2 >
1/g(\lambda_+^+)$. Let $\rho_i$ be the unique real number $>
\lambda_+$ satisfying $\omega_i^2 g(\rho_i) = 1$ for
$i=1,\cdots,q$. Then 
$$
\hat\lambda_{j_1+\cdots+j_{i-1}+\ell,n}
\xrightarrow[n\to\infty]{\text{a.s.}}  \rho_i
$$ 
for $i=1,\cdots,q$ and
$\ell=1,\cdots, j_i$ while 
$$\hat\lambda_{j_1+\cdots+j_{q}+1,n}
\xrightarrow[n\to\infty]{\text{a.s.}} \lambda_+\ .
$$
\end{theorem}
In the case where $\sqrt{n} X$ is a standard Gaussian matrix, $\pi$ is the 
Mar\v cenko-Pastur distribution with support $\support(\pi) = 
[\lambda_-, \lambda_+] = [(1-\sqrt{c})^2, (1+\sqrt{c})^2 ]$, and 
\begin{equation} 
\label{st-mp} 
m(x) = \frac{1}{2cx} \left( 1-c-x + 
\sqrt{(1-c-x)^2 - 4cx} \right) 
\end{equation} 
for $x \in (\lambda_+, \infty)$. After a few derivations, we obtain: 
\begin{corollary}
\label{mp}
Assume $\sqrt{n} X$ is standard Gaussian. Let $q$ be the maximum index such
that $\omega_q^2 > \sqrt{c}$. Then 
\[
\hat\lambda_{j_1+\cdots+j_{i-1}+\ell,n} \xrightarrow[n\to\infty]{\text{a.s.}} 
\frac{(\omega_i^2 +1)(\omega_i^2 + c)}{\omega_i^2} \quad 
\text{for} \ i=1,\ldots,q\ , 
\]
and 
$\hat\lambda_{j_1+\cdots+j_{q}+1,n} 
\xrightarrow[n\to\infty]{\text{a.s.}} (1+\sqrt{c})^2$. 
\end{corollary} 

We now turn our attention to the eigenspaces of the isolated eigenvalues. 

\subsection*{Asymptotic behavior of certain bilinear forms.} 
Recall the definition of $s$ as provided in Assumption \ref{ass:P-rough}.  
Given $i \leq s$, assume that $\omega_i^2 > 1/g(\lambda_+^+)$. 
Given two $N \times 1$ deterministic 
sequences of vectors $b_{1,n}$ and $b_{2,n}$ with bounded norms, we shall 
find here a simple asymptotic relation between 
$b_{1,n}^* \widehat\Pi_{i,n} b_{2,n}$ and $b_{1,n}^* \Pi_{i,n} b_{2,n}$, 
that will be at the basis of the Spike MUSIC algorithm. 
% (notice that $\widehat\Pi_{i,n}$ and $\Pi_{i,n}$ are well defined for $n$
% large enough by Theorem \ref{cvg-spk} and by {\bf A5}). \\
A close problem has been considered in \cite{ben-rao-11}. We consider here
a different technique, based on a contour integration and on the use of 
Lemmas \ref{lm-fq-haar} and \ref{lm-fq-two-haars}. This method lends itself
easily to the first and second order analyses of the Spike MUSIC algorithm
that we shall develop in the following sections. 

Writing 
${\bf b}_{i} = \begin{bmatrix} b_{i} \\ 0 \end{bmatrix}$ with $i=1,2$, we have
by virtue of Lemma \ref{lm-gram-hermitian}: 
\[
b_1^* \widehat\Pi_{i} b_2 
= 
\frac{-1}{\imath \pi} \oint_{{\mathcal C}_{i,n}}
{\bf b}_1^* \left( {\bs \Sigma} - z I \right)^{-1} {\bf b}_2 \ dz \ ,
\]
where ${\mathcal C}_{i,n}$ is a positively oriented circle that encloses the 
only singular values $\sqrt{\hat\lambda_{k,n}}$ of $\Sigma_n$ for which 
$i(k) = i$. 
Recalling \eqref{eq-structure-Sigma} and using Woodbury's identity 
(\cite[\S 0.7.4]{HorJoh94}) together with the fact that $J = J^{-1}$, we 
obtain:
\begin{multline}
b_1^* \widehat\Pi_{i} b_2 = 
\frac{-1}{\imath \pi} \oint_{{\mathcal C}_i} {\bf b}_1^* 
{\bf Q}(z) {\bf b}_2 \ dz 
\\
+ 
\frac{1}{\imath \pi} \oint_{{\mathcal C}_i} 
{\bf b}_1^* {\bf Q}(z)B\left( J + B^* {\bf Q}(z) B \right)^{-1}
B^* {\bf Q}(z) {\bf b}_2 \ dz \ .
\nonumber 
% \label{eq-proj-contour-integral} 
\end{multline} 
Using \eqref{eq-resolvent}, we obtain after a straightforward calculation: 
\begin{equation} 
b_{1,n}^* \widehat\Pi_{i,n} b_{2,n} = 
\frac{-1}{\imath \pi} \oint_{{\mathcal C}_{i,n}} {\bf b}_{1,n}^* 
{\bf Q}_n(z) {\bf b}_{2,n} \ dz 
+ 
\frac{1}{\imath \pi} \oint_{{\mathcal C}_{i,n}} 
\hat a_{1,n}^*(z) \widehat H_n(z)^{-1} \hat a_{2,n}(z) \, dz  
\label{eq-expansion-aH}
\end{equation}
where\footnote{Notice that $\hat a_{\ell,n}^*(z)$ as defined is {\bf
    not} the Hermitian adjoint of $\hat a_{\ell,n}(z)$. Despite this
  ambiguity, we introduce this notation which remains natural and
  widespread in Signal Processing. \label{fn:not-hermition}}
\begin{eqnarray}
\hat a_{\ell,n}(z) &=& 
\begin{bmatrix}
z U_n^* Q_n(z^2) \\ \Omega_n V_n^* \widetilde Q_n(z^2) X_n^* \end{bmatrix} 
b_{\ell,n} \ ,\nonumber \\
\hat a_{\ell,n}^*(z) &=& 
b_{\ell,n}^* 
\begin{bmatrix} z Q_n(z^2) U_n & X_n \widetilde Q_n(z^2) V_n \Omega_n 
\end{bmatrix}\ . \label{eq:def-a-hat-star}
\end{eqnarray}
Intuitively, the first integral is zero for $n$ large enough and the second
is close to 
\[
T_{i,n} = 
\frac{1}{\imath \pi} \oint_{\gamma_i} 
a_{1,n}^*(z) H(z) ^{-1} a_{2,n}(z) \, dz \ ,
\]  
where $\gamma_i$ is a small enough positively oriented circle which does not 
meet the image of $\support(\pi)$ by $x\mapsto \sqrt{x}$ nor any of the 
$\sqrt{\rho_\ell}$ and such that only $\sqrt{\rho_i} \in \interior(\gamma_i)$,  
the interior of the disk defined by $\gamma_i$ (see Figure \ref{fig:contour}), 
$a_{\ell,n}^*(z)=b_{\ell,n}^* \begin{bmatrix} z m(z^2) U_n & 0 \end{bmatrix}$,
and 
\[
a_{\ell,n}(z) = \begin{bmatrix} z m(z^2) U_n^* \\ 0 \end{bmatrix} b_{\ell,n} 
. 
\]

\begin{figure}[t]
  \begin{center}
    \includegraphics[width=0.7\linewidth]{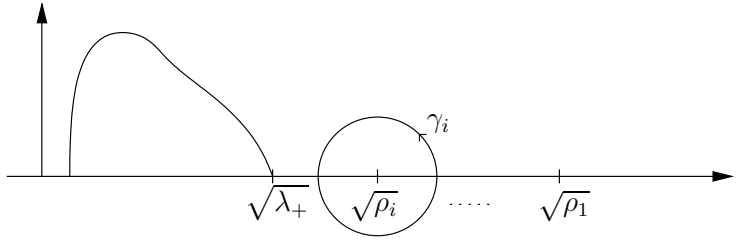}
  \end{center}
  \caption{The contour $\gamma_i$ w.r.t. the support of the limit
singular value distribution of $X_n$ and the other $\sqrt{\rho_\ell}$.}
  \label{fig:contour}
\end{figure}

The approximation $b_1^* \widehat\Pi_i b_2 \simeq T_i$ will be justified 
rigorously below. 
For the moment, let us develop the expression of $T_i$. Defining the
$r\times r$ matrices:
\[
{\mathcal I}_{i} = \begin{bmatrix} 0 & & \\ & I_{j_i} & \\ & & 0  
\end{bmatrix}\ ,
\]
where the integers $j_i$ are defined in Assumption \ref{ass:P-rough}, we have 
\begin{equation}
\label{inv(H)} 
H(z)^{-1} = \sum_{i=1}^s 
\frac{1}
{z^2 m(z^2) \tilde m(z^2) \omega_i^2 -1} 
\begin{bmatrix} z \tilde m(z^2) \omega_i^2 & - 1 \\ 
- 1 & z m(z^2) \end{bmatrix} 
\otimes {\mathcal I}_i \ ,
\end{equation} 
% Writing $U = \begin{bmatrix} u_1 \cdots u_r \end{bmatrix}$, we then have 
which leads to 
\begin{align} 
T_i 
% \frac{1}{\imath \pi} 
% \sum_{\ell=1}^r \oint_{\partial B(\sqrt{\rho_k}, R)} 
% \frac{
% \begin{bmatrix} z m(z^2) b_1^* u_\ell & 0 \end{bmatrix} 
% \begin{bmatrix} z \tilde m(z^2) \omega_\ell^2 & - 1 \\ 
% - 1 & z m(z^2) \end{bmatrix} 
% \begin{bmatrix} z m(z^2) u_\ell^* b_2 \\ 0 \end{bmatrix}}
% {z^2 m(z^2) \tilde m(z^2) \omega_\ell^2 -1} 
%  \, dz \nonumber \\
&= 
\frac{1}{\imath \pi} 
\sum_{\ell=1}^s b_1^* \Pi_\ell b_2 \oint_{\gamma_i} 
\frac{z^3 m(z^2)^2 \tilde m(z^2) \omega_\ell^2}
{z^2 m(z^2) \tilde m(z^2)\omega_\ell^2 - 1} 
\, dz \nonumber \\
&= 
\frac{1}{2 \imath \pi} \sum_{\ell=1}^s   b_1^* \Pi_\ell b_2 
\oint_{\gamma_i'}
\frac{w m(w)^2 \tilde m(w) \omega_\ell^2}
{w m(w) \tilde m(w)\omega_\ell^2 - 1} 
\, dw \nonumber 
\end{align} 
by making the change of variable $w = z^2$. 
Observe that the path $\gamma_i'$ now encloses $\rho_i$ only. 
Recall that $w m(w) \tilde m(w) \omega_\ell^2 - 1 = 0$ if and only if 
$w = \rho_\ell$ for every $\ell$ such that $\omega_\ell^2 > 1/g(\lambda_+^+)$, 
and since $g(w) = 
w m(w) \tilde m(w)$ is decreasing on $(\lambda_+, \infty)$, these zeros are 
simple. As a result, the integrals above are 
equal to zero for $\ell \neq i$, and the integrand has a simple pole at 
$w = \rho_i$ for $\ell= i$. By the Residue Theorem, we have:
\begin{equation} 
T_i = 
\frac{1}{\imath \pi} \oint_{\gamma_i} 
a_{1}^*(z) H(z) ^{-1} a_{2}(z) \, dz 
= 
\frac{\rho_i m(\rho_i)^2 \tilde m(\rho_i)}
{(\rho_i m(\rho_i) \tilde m(\rho_i))'} 
b_1^* \Pi_i b_2 
\label{eq-T_k} 
\end{equation} 
where the denominator at the right hand side is the derivative of 
the function $\lambda \mapsto \lambda m(\lambda) \tilde m(\lambda)$ 
at $\lambda = \rho_i$. 
We now make this argument more rigorous:

\begin{theorem}
\label{bilin-forms}
Let Assumptions \ref{ass:asymptotics}-\ref{ass:P-rough} hold true. For
a given $i\leq s$, assume that $\omega_i^2 > 1/g(\lambda_+^+)$. Let
$(b_{1,n})$ and $(b_{2,n})$ be two sequences of deterministic vectors
with bounded norms. Then
\[
b_{1,n}^* \widehat\Pi_{i,n} b_{2,n} - 
\frac{\rho_i m(\rho_i)^2 \tilde m(\rho_i)}
{(\rho_i m(\rho_i) \tilde m(\rho_i))'} 
b_{1,n}^* \Pi_{i,n} b_{2,n}  
\xrightarrow[n\to\infty]{\text{a.s.}} 0 \ .
\]
\end{theorem} 
\begin{proof}
% To start with, we restrain ourselves to the set ${\mathcal O}_n$ and we replace 
% the random circle in \eqref{eq-expansion-aH} with a deterministic one. 
Write 
\[
\widehat{T}_i = 
\frac{1}{\imath \pi}
\oint_{\gamma_i}  
\hat a_1^*(z) \widehat H(z)^{-1} \hat a_2(z) \, dz \ .
\] 
% where $\gamma_i$ is a small enough positively oriented circle which does not 
% meet $\support{\pi} \cup \{ \sqrt{\rho_1}, \ldots, \sqrt{\rho_s} \}$ and such 
% that only $\sqrt{\rho_i} \in \interior(\gamma_i)$. 
Then, with probability one, $b_1^* \widehat\Pi_i b_2 = \widehat{T}_i$
for $n$ large enough. Indeed, on the set $O_n$ (as defined in
\eqref{eq:def-O}), the singular values of $\Sigma$ greater than
$\sqrt{\lambda_+}+ \varepsilon_1$ coincide with the poles of $\widehat
H(z)$ which are greater than $\sqrt{\lambda_+} +\varepsilon_1 $ by the
argument preceding Theorem \ref{cvg-spk}. On this set, the first
integral on the right hand side (r.h.s.) of \eqref{eq-expansion-aH} is zero,
and by Theorem \ref{cvg-spk}, the second integral can be replaced with
$\int_{\gamma_i}$ with probability one for $n$ large enough.  By Lemma
\ref{lm-uniform-cvg}, the differences $\widehat H(z) - H(z)$, $\hat
a_1(z) - a_1(z) $, and $\hat a_2(z) - a_2(z)$ a.s.~converge to zero,
uniformly on $\gamma_i$. Hence $\widehat T_i - T_i \toas 0$.
\end{proof}

\section{The Spike MUSIC Estimation Algorithm} 
\label{spk-music} 

\subsection*{Algorithm description} 
We now consider the application context described in the introduction, and
assume that $P_n = B_n(\varphi_1,\ldots,\varphi_r) S_n^*$ where 
$B_n(\varphi_1,\ldots,\varphi_r) = \begin{bmatrix} 
b_n(\varphi_1) \cdots b_n(\varphi_k) \end{bmatrix}$, and 
$b_n(\varphi) = N^{-1/2} \begin{bmatrix} \exp( - \imath D \ell \varphi ) 
\end{bmatrix}_{\ell=0}^{N-1}$ with domain $\varphi \in [0, \pi/D]$.
When the $\varphi_k$ are different, one can check that $B_n^* B_n \to
I_r$ as $n\to\infty$.  In most practical cases of interest, $S_n^* S_n
\to O^2$ where $O$ is given by Equation \eqref{eq-limit-O}. In these
conditions, due to $B_n^* B_n \to I_r$, the diagonal elements of $O$
are the limits of the singular values of $P_n$ and Assumption
\ref{ass:P-rough} holds true. 

In the area of signal processing, the positive real numbers $\omega_i^2$ are 
called the Signal to Noise Ratios (SNR) associated with the $r$ sources. 
Assumption \ref{ass:P-rough} becomes: 

\begin{assumption}\label{ass:P-finer}
Matrices $P_n$ of dimension $N\times n$ are deterministic and are written:
$$
P_n = B_n(\varphi_1, \cdots \varphi_r) S_n^*
$$
 where $r$ is a fixed integer, $B_n(\varphi_1, \cdots \varphi_r) = 
\begin{bmatrix} b_n(\varphi_1) &\cdots & b_n(\varphi_r) \end{bmatrix}$ is a $N\times r$ matrix, 
$b_n(\varphi) = N^{-1/2} \begin{bmatrix} \exp( - \imath D \ell \varphi ) 
\end{bmatrix}_{\ell=0}^{N-1}$ on $\varphi \in [0, \pi/D]$, and the $\varphi_k$
are all different. Matrix $S_n$ of dimensions $n\times r$ satisfies:
$$
\sqrt{n} ( S_n^* S_n - O^2 ) = {\mathcal O}(1) 
$$ 
as $n\to\infty$, where $O$ is defined in Assumption \ref{ass:P-rough}, 
and ${\mathcal O}$ is the classical Landau notation. 
\end{assumption}
The assumption over the speed of convergence of $S^* S$ will be needed only 
for the purpose of the second order analysis. It is satisfied by most 
practical systems met in the field of signal processing. We moreover observe 
that it is possible to relax the assumption that $O$ is diagonal at the 
expense of a more complicated second order analysis. 

In order for the algorithm to be able to estimate the $r$ angles, it is 
necessary that the perturbation $P$ gives rise to $r$ isolated eigenvalues, 
a fact that is stated in the following assumption:
\begin{assumption}\label{ass:isolated}
  Recall the definition \eqref{eq:def-g} of function $g$, let
  $\lambda_+$ as defined in \ref{ass:resolvent} and let
  $g(\lambda_+^+)=\lim_{x\downarrow \lambda_+} g(x)$. Let the
  $\omega_i$'s as defined in \ref{ass:P-rough}, then: 
$$
\omega_r^2 > \frac{1} {g(\lambda_+^+)} \ .
$$
\end{assumption}

The Spike MUSIC algorithm goes like this. The localization function 
$\chi(\varphi)$ defined in the introduction is also written as 
$\chi(\varphi) = \sum_{i=1}^s b(\varphi)^* \Pi_i b(\varphi)$. Given $\varphi$,
the results of the previous section (Theorems \ref{cvg-spk} and 
\ref{bilin-forms} with $b_1 = b_2 = b(\varphi)$) show us that: 
\begin{equation}\label{eq:def-chi-hat}
\hat\chi_n(\varphi) = 
\sum_{k=1}^r 
| b_n(\varphi)^* \hat u_{k,n} |^2 \zeta(\hat\lambda_{k,n})  \ ,
\end{equation}
where 
\begin{equation}\label{eq:def-zeta}
\zeta(\lambda) = 
\frac{(\lambda m(\lambda) \tilde m(\lambda))'} 
{\lambda m(\lambda)^2 \tilde m(\lambda)} 
\end{equation}
is a consistent estimator of $\chi_n(\varphi)$ in the asymptotic regime
described by \ref{ass:asymptotics}. By searching for the maxima of $\hat\chi(\varphi)$,
we infer that we obtain consistent estimates of the angles or arrival. 
Observe that this algorithm requires the knowledge of the 
Stieltjes Transform of the limit spectral measure of $XX^*$ (available if
the statistical description of the noise is known) and the number 
$r$ of emitting sources. Notice that when this number is unknown, it can
be estimated along the ideas described in \emph{e.g.} 
\cite{BianNajMai'11, nad-jmva11}. \\ 
We now perform the first order analysis of this algorithm. 

\subsection*{First order analysis of the Spike MUSIC algorithm}
We now formalize the argument of the previous paragraph and we push it further
to show the consistency ``up to the order $n$'' of the Spike MUSIC estimator.
We shall need this speed to perform the second order analysis 
(Lemma \ref{denom} below). 
\begin{theorem} 
\label{n-consist} 
Let Assumptions \ref{ass:asymptotics}-\ref{ass:P-finer} hold true. 
Then for all $k=1,\cdots,r$, there
exists a local maximum $\hat\varphi_{k,n}$ of $\hat\chi_n(\varphi)$ such that 
\[
n(\hat\varphi_{k,n} - \varphi_k) \xrightarrow[n\to\infty]{\text{a.s.}} 0 . 
\]
\end{theorem} 
The proof of this theorem is performed in two steps. With an approach similar 
to the one used in Section \ref{vap-spikes}, we first prove that 
$\hat\chi(\varphi) - \chi(\varphi) \toas 0$, and the convergence
is uniform on $\varphi\in[0, \pi/D]$ (Proposition ~\ref{prop-uniform-cvg} below). Next, 
following the technique of \cite{han-jap71, han-jap73}, we prove that this 
uniform a.s.~convergence leads to Theorem \ref{n-consist}. 

In the sequel, we write:
\begin{eqnarray}
\label{eq-a-hat-a} 
\hat a(z, \varphi) &=& 
\begin{bmatrix}
z U^* Q(z^2) \\ \Omega V^* \widetilde Q(z^2) X^* \end{bmatrix} b(\varphi)
\quad \text{and} \quad  
a(z, \varphi) \ =\  
\begin{bmatrix}
z m(z^2) U^* \\ 0 \end{bmatrix} b(\varphi)\ , \\
\hat a^*(z, \varphi) &=& b^*(\varphi)
[
z Q (z^2) U \quad   X \widetilde Q(z^2) V \Omega ]\nonumber  \ ,\\
a^*(z, \varphi) &=& 
b(\varphi) [
z m(z^2) U \quad   0] \ .\nonumber
\end{eqnarray} 
Beware that $\hat a ^*$ and $a^*$ are not the Hermitian adjoints of $\hat a$ and $a$ (see the footnote associated to 
Eq. \eqref{eq:def-a-hat-star}). 
% \begin{lemma}
% \label{lm-random-deterministic}
% Assume the setting of Theorem \ref{n-consist}. For $i=1,\ldots,s$, write 
% \[
% \hat{I}_i(\varphi) = 
% \1_{{\mathcal O}_n} \times 
% \frac{1}{\imath \pi}
% \oint_{\gamma_i}  
% \hat a^*(z, \varphi) \widehat H(z)^{-1} \hat a(z, \varphi) \, dz 
% \] 
% where $\gamma_i$ is a small enough positively oriented circle which does not 
% meet $\support{\pi} \cup \{ \sqrt{\rho_1}, \ldots, \sqrt{\rho_s} \}$ and such 
% that only $\sqrt{\rho_i} \in \interior(\gamma_i)$. Then 
% \[
% \max_i \max_{\varphi} \left|  
% \hat{I}_i(\varphi) - b(\varphi)^* \widehat\Pi_i b(\varphi) \right| 
% \xrightarrow[n\to\infty]{\text{a.s.}} 0 
% \]
% \end{lemma}
% \begin{proof}
% Recall that $b^*(\varphi) \widehat\Pi_i b(\varphi)$ is given by 
% \eqref{eq-expansion-aH} with $b_1 = b_2 = b(\varphi)$. 
% % Obviously, $\max_\varphi b^*(\varphi) \widehat\Pi_i b(\varphi) 
% (1 - \1_{{\mathcal O}_n}) \to 0$ almost surely. 
% On ${\mathcal O}_n$, the singular 
% values of $\Sigma$ greater than $\sqrt{\lambda_+}$ coincide with the poles of
% $\widehat H(z)$ which are greater than $\sqrt{\lambda_+}$ by the argument 
% preceding Theorem \ref{cvg-spk}. Thanks to this theorem, the first
% integral at the right hand side of \eqref{eq-expansion-aH} is zero and the
% second can be replaced with $\int_{\gamma_i}$ for $\omega\in {\mathcal O}_n$ and 
% for $n$ large enough, hence the result. 
% \end{proof} 
\begin{proposition}
\label{prop-uniform-cvg}
In the setting of Theorem \ref{n-consist}, 
$$
\max_{\varphi\in [0, \pi/D]} 
\left| \hat\chi_n(\varphi) - \chi_n(\varphi)\right|
\xrightarrow[n\to\infty]{\text{a.s.}} 0 \ .
$$ 
\end{proposition}

\begin{proof}
Write 
\[
\hat\chi(\varphi) - \chi(\varphi) = 
\sum_{k=1}^r (\zeta(\hat\lambda_k) - \zeta(\rho_{i(k)}))
| b(\varphi)^* \hat u_k |^2 + 
\sum_{i=1}^s \left( \zeta(\rho_{i}) b(\varphi)^* \widehat\Pi_i b(\varphi) 
- b(\varphi)^* \Pi_i b(\varphi) \right)  . 
\] 
By Theorem \ref{cvg-spk} and the continuity of $\zeta$ on
$( \lambda_+, +\infty )$, the first term at the r.h.s. goes to zero
a.s. and uniformly in $\varphi$. Consider the second term. 
Let $\gamma_i$ be a small enough positively oriented circle which does not 
meet $\support(\pi) \cup \{ \sqrt{\rho_1}, \cdots, \sqrt{\rho_s} \}$ and such 
that only $\sqrt{\rho_i} \in \interior(\gamma_i)$. 
Since $\hat\lambda_k \toas \rho_{i(k)}$, 
% By an argument similar to the one used in the proof of Theorem 
% \ref{bilin-forms}, we can show that 
$$\max_i \max_{\varphi} \left|  
b(\varphi)^* \widehat\Pi_i b(\varphi) - \widehat{T}_i(\varphi) \right| 
= 0$$ 
a.s. for $n$ large enough, where 
\[
\widehat{T}_i(\varphi) = 
\frac{1}{\imath \pi}
\oint_{\gamma_i}  
\hat a^*(z, \varphi) \widehat H(z)^{-1} \hat a(z, \varphi) \, dz  
\] 
Recalling Eq. \eqref{eq-T_k}, it will
therefore be enough to prove that
$$
\max_{1\le i\le s} \max_{\varphi\in[0,\pi/D]} 
| Z_{i}(\varphi) | \xrightarrow[n\to \infty ]{a.s.} 0 \ ,
$$
where
$$
Z_{i}(\varphi) = \frac{1}{\imath \pi}
\oint_{\gamma_i} 
\left( \hat a^*(z, \varphi) \widehat H(z)^{-1} \hat a(z, \varphi) 
- a^*(z, \varphi) H(z)^{-1} a(z, \varphi) \right) 
\, dz \ .
$$
We have 
\[
\max_\varphi | Z_{i}(\varphi) | \leq 
2 R \int_0^1 
\max_\varphi e(\sqrt{\rho_i} + R e^{2 \imath \pi \theta} ,\varphi) 
\, d\theta 
\]
where $R$ is the radius of $\gamma_i$ and where 
\begin{align*} 
e(z,\varphi) &= 
\left| 
\hat a^*(z, \varphi) \widehat H(z)^{-1} \hat a(z, \varphi) 
- a^*(z, \varphi) H(z)^{-1} a(z, \varphi) \right| . \\ 
&\leq 
\left| (\hat a^* - a^*) H^{-1} \hat a \right| + 
\left| a H^{-1} (\hat a - a) \right| + 
\left| \hat a^* ( \widehat H^{-1} - H^{-1} ) \hat a \right| . 
\end{align*} 
Since $\| H^{-1} \|$, $\max_\varphi \| a \|$ and $\max_\varphi \| \hat a \|$ 
are bounded on $\gamma_i$, $e(z,\varphi)$ satisfies on this path 
\[
e(z,\varphi) \leq K \left( \| \hat a(z,\varphi) - a(z,\varphi) \| 
+ \| \widehat H(z)^{-1} - H(z)^{-1} \|  \right) \ .
\] 
By Lemma \ref{lm-uniform-cvg} and the fact that $\| H^{-1} \|$ is bounded
on $\gamma_i$, 
the term $\| \widehat H^{-1} - H^{-1} \| = 
\| \widehat H^{-1} ( H - \widehat H) H^{-1} \|$ converges to zero uniformly 
on $\gamma_i$ with probability one. To obtain the result, we prove that 
$\| \hat a - a \| \toas 0$ and that this convergence is uniform on 
$(z, \varphi) \in  \gamma_i \times [0, \pi/D]$.
Let us focus on the first term $z u_1^* (Q(z^2) - m(z^2) I ) b(\varphi)$
of $\hat a - a$, where we recall that $u_1$ is the first column of $U$. 
Since $\| b(\varphi) \| = \| u_1 \| = 1$, 
\[ 
| z u_1^* (Q(z^2) - m(z^2) I ) b(\varphi) | \leq 
| z u_1^* (Q(z^2) - \alpha(z^2) I ) b(\varphi) | + 
| z( \alpha(z^2) - m(z^2) ) |\ .
\]
With probability one, the second term converges to zero on  
$\gamma_i$, and the convergence is uniform (along the 
principle of the proof of Lemma \ref{lm-uniform-cvg}). 
Since 
$$
\sup_n \max_\varphi \| n^{-1} b'(\varphi) \| 
= 
\sup_n \max_\varphi \left\| n^{-1} N^{-1/2} \begin{bmatrix} 
\ell D \exp(-\imath \ell D \varphi) \end{bmatrix}_{\ell=0}^{N-1} \right\|
< \infty\ ,
$$
the term 
\[
\xi(z, \varphi) = \1_{O_n} \times z u_1^* (Q(z^2) - \alpha(z^2) I) 
b(\varphi)
\]
satisfies 
\[
| \xi(z_1, \varphi_1) - \xi(z_2, \varphi_2) | \leq 
K( n |\varphi_1 - \varphi_2| + | z_1 - z_2 | ) 
\]
for every $(z_1, \varphi_1)$, $(z_2, \varphi_2)$ in 
$\gamma_i \times [0, \pi/D]$. Therefore, it will be enough to prove that
\[
\max_{(z, \varphi) \in A_n \times B_n} \xi(z, \varphi) 
\xrightarrow[n\to\infty]{\text{a.s.}} 0 
\]
where $A_n$ contains $n$ regularly spaced points in $\gamma_i$ and $B_n$ 
contains $n^2$ regularly spaced points in $[0, \pi/D]$. This can be obtained 
from Lemma \ref{lm-fq-haar} with $p=9$, Markov inequality and Borel Cantelli's 
lemma. The other terms of $\hat a - a$ can be handled similarly. 
\end{proof}

We now prove Theorem \ref{n-consist} by following the ideas of 
\cite{han-jap71, han-jap73}. To that end, we need the following lemma, proven 
in \cite{cib-lou-ser-gia-it02}: 
\begin{lemma} 
\label{ciblat} 
Let $(c_N)$ be a sequence of real numbers belonging to a compact of 
$[-1/2, 1/2]$ and converging to $c$. Let 
\[
q_N(c_N) = \frac 1N \sum_{k=0}^{N-1} \exp(- 2 \imath \pi k c_N ) \ .
\]
Then the following hold true:
\begin{eqnarray*}
q_N(c_N) &\xrightarrow[N\to\infty]{}& 0 \quad \text{if} \ c \neq 0\ , \\
q_N(c_N) &\xrightarrow[N\to\infty]{}& 0 \quad \text{if} \ c = 0 \ \text{and} \ 
N|c_N - c| \to \infty\ , \\ 
q_N(c_N) &\xrightarrow[N\to\infty]{}& \exp(-\imath \pi d) \sinc(d) 
\quad \text{if} \ c = 0 \ \text{and} \ N|c_N - c| \to d\ , 
\end{eqnarray*}
where $\sinc$ stands as usual for sine cardinal.
\end{lemma}

\begin{proof}[Proof of Theorem \ref{n-consist}] 
We start by observing that $\chi(\varphi) = d(\varphi)^* (B^* B)^{-1}
d(\varphi)$ where $B$ is the matrix defined in \ref{ass:P-finer} and
where $d(\varphi) = \begin{bmatrix} b(\varphi_k)^*
  b(\varphi) \end{bmatrix}_{k=1}^r$.  By Lemma \ref{ciblat}, $B^* B
\to I_r$, hence $\chi(\varphi) -
\| d(\varphi) \|^2 \to 0$. 

In the remainder of the proof, we shall stay in the probability one
set where the uniform convergence in the statement of Proposition
\ref{prop-uniform-cvg} holds true.  Taking $k=1$ without loss of
generality, we shall show that any sequence $\hat\varphi_{1,n}$ for
which $\hat\chi(\hat\varphi_{1,n})$ attains its maximum in the closure
of a small neighborhood of $\varphi_1$ satisfies $N(\hat\varphi_{1,n}
- \varphi_1) \to 0$. Given a sequence of such $\hat\varphi_{1,n}$,
assume we can extract a subsequence $\hat\varphi_{1,n^*}$ such that
$N| \hat\varphi_{1,n^*} - \varphi_1 | \to \infty$.  In this case,
Lemma \ref{ciblat} and the observations made above on the structure of
$\chi(\varphi)$ show that $\chi(\hat\varphi_{1,n^*}) \to 0$.  Since
$\max_\varphi | \hat\chi(\varphi) - \chi(\varphi) | \to 0$,
$\hat\chi(\hat\varphi_{1,n^*}) \to 0$. But $\hat\chi(\varphi_{1}) \to
\chi(\varphi_{1}) = 1$, which contradicts the fact that
$\hat\varphi_{1,n^*}$ maximizes $\hat\chi$.  Hence the sequence $N
(\hat\varphi_{1,n^*} - \varphi_1)$ belongs to a compact.  Assume $N
(\hat\varphi_{1,n^*} - \varphi_1) \not\to 0$.  If we take a further
subsequence of the latter that converges to a constant $d \neq 0$,
then by Lemma \ref{ciblat}, $\hat\chi$ converges to $\sinc(d)^2 < 1$
along this subsequence, which also raises a contradiction. This proves
the theorem.\end{proof}

\section{Second Order Analysis of the Spike MUSIC Estimator} 
\label{sec-clt}

In order to perform the second order analysis, we also assume: 
\begin{assumption}\label{ass:strong-conv}
Let $\lambda_-$, $\lambda_+$, $\alpha$ and $m$ be as in \ref{ass:resolvent}.
Then for any $z \in \CC - [\lambda_-, \lambda_+]$, 
$\sqrt{n} \left( \alpha(z) - m(z) \right)$ converges in probability to zero. 
\end{assumption}

\begin{remark}
\label{rq-2order} 
If $\sqrt{n} X$ is standard Gaussian and if $c_n = N/n$ satisfies 
$\sqrt{n}(c_n - c) \to 0$, then Assumption \ref{ass:strong-conv} is satisfied. 
Indeed, call $m_n(z)$ the Stieltjes Transform of the
Mar\v cenko-Pastur distribution, \emph{i.e.}, the analytic continuation of
\eqref{st-mp}, when $c$ is replaced with $c_n$, 
and let $\pi_n$ be the associated probability measure. 
For $z \in \CC - [\lambda_-, \lambda_+]$, function $f(x) = (x-z)^{-1}$ is 
analytic 
outside the support of $\pi_n$ for $n$ large, and \cite[Th.1.1]{sil-bai-ap04} 
can be applied to show that $\sqrt{n} ( \alpha_n(z) - m_n(z) ) \toP 0$.
When $\sqrt{n}(c_n - c) \to 0$, it is furthermore clear that 
$\sqrt{n} ( m_n(z) - m(z) ) \to 0$.
\end{remark} 

The main result of this section is the following: 
\begin{theorem}
\label{2nd-order} 
Let Assumptions \ref{ass:asymptotics}-\ref{ass:strong-conv} hold true. 
Then the estimates $\hat\varphi_{k,n}$ satisfy 
\begin{equation}
\label{clt} 
n^{3/2}\begin{bmatrix} 
\hat\varphi_{1,n} - \varphi_1 \\
\vdots \\
\hat\varphi_{r,n} - \varphi_r 
\end{bmatrix} 
\xrightarrow[n\to\infty]{{\mathcal D}} 
{\mathcal N}\left(0, \begin{bmatrix} \sigma_1^2 I_{j_1} & & \\ &\ddots& \\ 
& & \sigma_s^2 I_{j_s} \end{bmatrix} \right) 
\end{equation} 
where 
\[ 
\sigma_i^2 = \frac{6}{c^2 D^2} \left( 
\frac{m'(\rho_i) - m(\rho_i)^2}{c m(\rho_i)^2} + 
\omega_i^2 \left( m(\rho_i) + \rho_i m'(\rho_i) \right) \right), \quad 
1\le i\le s\ . 
\]
\end{theorem} 
When $\sqrt{n} X$ is standard Gaussian, plugging the r.h.s. of 
\eqref{st-mp} into this expression leads after some derivations to: 
\begin{corollary}
\label{var-mp}
If $\sqrt{n} X$ is standard Gaussian and if $\sqrt{n}(c_n - c) \to 0$, 
the convergence \eqref{clt} holds true with 
\[
\sigma_i^2 = \frac{6}{c^2 D^2} \frac{\omega_i^2 +1}{\omega_i^4 - c} . 
\]
\end{corollary} 
This corollary calls for some comments: 

\begin{remark}[Efficiency at high SNR]
Recalling that $\omega_i^2 > \sqrt{c}$ 
is the condition for the existence of a corresponding isolated eigenvalue 
(Corollary 
\ref{mp}), we observe that the estimator variance for $\varphi_k$ goes to 
infinity as the corresponding $\omega_i^2$ decreases to $\sqrt{c}$. At the
other extreme, this variance behaves like $6 c^{-2} D^{-2} \omega_i^{-2}$ 
as $\omega_i^2 \to\infty$. It is useful to notice that this asymptotic 
variance coincides with the Cram\'er-Rao bound for estimating $\varphi_k$ 
\cite{sto-neh89}. In other words, the Spike MUSIC estimator is efficient at 
high SNR when the noise matrix is standard Gaussian. 
\end{remark}

\subsection*{A numerical illustration} 
In order to illustrate the convergence and the fluctuations of the 
Spike MUSIC algorithm, we simulate a radio signal transmission satisfying 
Assumptions \ref{ass:asymptotics}-\ref{ass:strong-conv}. 
We consider $r = 2$ emitting sources located at the angles $0.5$ and 
$1$ radian, and a number of receiving antennas
ranging from $N = 5$ to $N = 50$. The observation window length is 
set to $n = 2N$ (hence $c=0.5$). The noise matrix $X_n$ is such
that $\sqrt{n} X_n$ is standard Gaussian. The source powers are assumed 
equal, so that the matrix $O$ given by Equation \eqref{eq-limit-O} is
written $O = \omega I_2$, and the Signal to Noise Ratio for any source
is $\text{SNR} = 10 \log_{10} \omega^2$ decibels. 
In Figure \ref{var_vsN}, the SNR is set to $10$ dB, and the empirical variance 
of $\hat\varphi_{1,n} - \varphi_1$ (red curve) is computed over $2000$ runs. 
The variance provided by Corollary \ref{var-mp} 
% as well as the Cram\'er-Rao Bound are 
is also plotted versus $N$. We observe a good fit between the 
variance predicted by Corollary \ref{var-mp} and the empirical variance 
after $N = 15$ antennas. 
\begin{figure}[t]
  \begin{center}
    \includegraphics[width=0.7\linewidth]{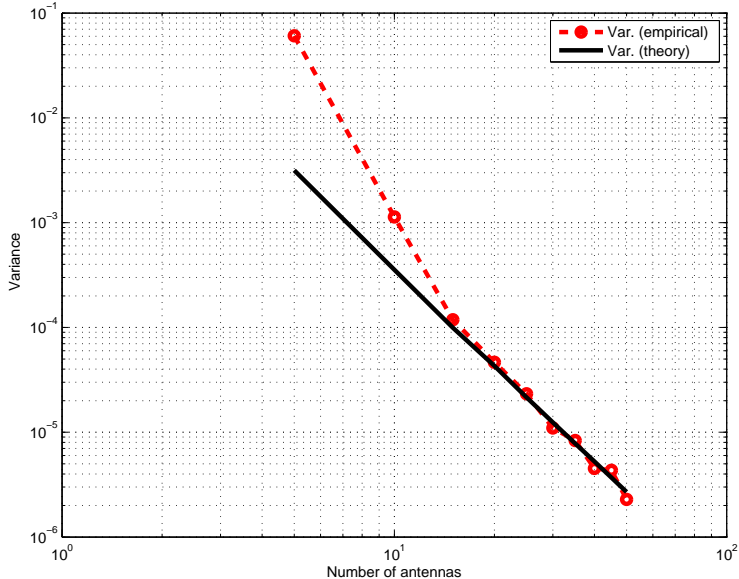}
  \end{center}
  \caption{Spike MUSIC algorithm, Variance \emph{vs} $N$.} 
  \label{var_vsN}
\end{figure}
In Figure \ref{var_vs_snr}, the variance is plotted as a function of the
SNR, the number of antennas being fixed to $N = 20$. The empirical 
variance is computed over $5000$ runs. The Cram\'er-Rao Bound is also
plotted. The empirical variance fits the theoretical one from 
$\text{SNR} \approx 6$ dB upwards. 
\begin{figure}[t]
  \begin{center}
    \includegraphics[width=0.7\linewidth]{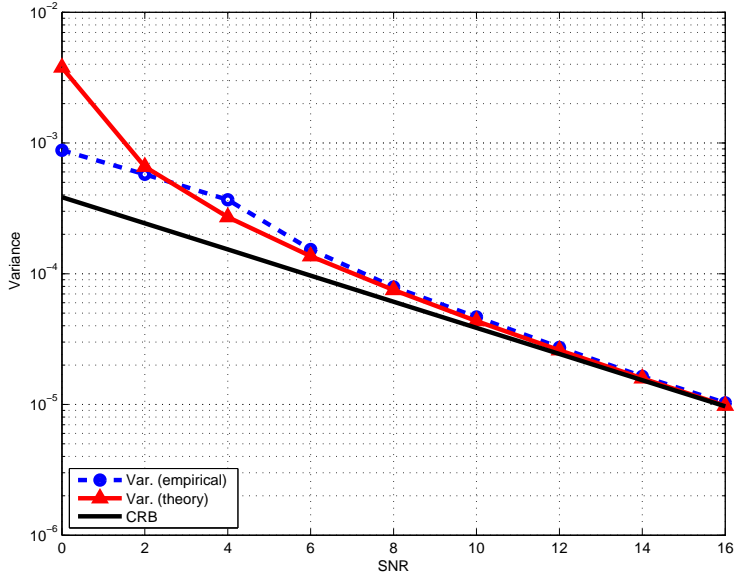}
  \end{center}
  \caption{Spike MUSIC algorithm, Variance \emph{vs} the SNR.} 
  \label{var_vs_snr}
\end{figure}

\subsection*{Proof of Theorem \ref{2nd-order}.} 
We start with some additional notations and definitions. Matrix 
$B = \begin{bmatrix} b(\varphi_1), \ldots, b(\varphi_r) \end{bmatrix}$ 
will be often written as $B = [ b_1, \ldots, b_r]$ or in block form as
$B= \begin{bmatrix} B_1, \ldots, B_s \end{bmatrix}$ where $B_i$ has $j_i$ 
columns. We shall also write 
$B' = \begin{bmatrix} b'(\varphi_1), \ldots, b'(\varphi_r) \end{bmatrix}$ 
and 
$B'' = \begin{bmatrix} b''(\varphi_1), \ldots, b''(\varphi_r) \end{bmatrix}$
where $b'(\varphi)$ and $b''(\varphi)$ are respectively the first and
second derivatives of $b(\varphi)$. We shall also use the short hand notations
$B' = [b_1', \ldots, b_r']$ and $B'' = [b_1'', \ldots, b_r'']$. 
Matrix $B^\perp = [ b_1^\perp, \ldots, b_r^\perp]$ will be defined by the equation
\begin{equation}\label{eq:def-B-prime}
\frac{1}{n} B' = - \frac{\imath cD}{2} B + \frac{cD}{2 \sqrt{3}} B^\perp . 
\end{equation}
Finally, if $x_n,y_n$ are random sequences, we denote by $x_n \asymp y_n$ 
the convergence $x_n - y_n \xrightarrow[]{\mathcal P}0$.  

We now state some preliminary results. In the following, we say that the
complex random vector $\eta$ is governed by the law ${\cal CN}(0,R)$ where
$R$ is a nonnegative Hermitian matrix if the real vector 
$\begin{bmatrix} \Re(\eta) \\ \Im(\eta) \end{bmatrix}$ has the law
${\cal N}\Bigl(0, \frac 12 \begin{bmatrix} \Re(R) & - \Im(R) \\
\Im(R) & \Re(R) \end{bmatrix}\Bigr)$. 
The following proposition, whose proof is postponed to \ref{clt-fq}, is 
crucial: 
\begin{proposition}
\label{clt-haar}
Let Assumptions \ref{ass:asymptotics}-\ref{ass:eigen-max} hold
true. Let $t \leq N$ be a fixed integer, let $W = \begin{bmatrix}
  w_1, \cdots, w_t \end{bmatrix}$ and $\widetilde W = \begin{bmatrix}
  \tilde w_1, \cdots, \tilde w_t \end{bmatrix}$ be deterministic
isometry matrices with dimensions $N \times t$ and $n \times t$
respectively. Let $\rho$ be a real number such that $\rho >
\lambda_+$. Then
\[ 
\xi_n = \sqrt{n} \left( 
 W^* \Bigl( Q(\rho) - \alpha(\rho) I_N \Bigr)W , \ 
\widetilde W^* \Bigl( \widetilde Q(\rho)  
- \tilde\alpha(\rho) I_n \Bigr) \widetilde W , \ 
W^* X \widetilde Q(\rho) \widetilde W 
\right) 
\] 
is tight. 

Assume $t$ is even. Given real numbers $\rho_1, \ldots, \rho_{t/2}$ all 
strictly greater than $\lambda_+$, the $t\times 1$ random vector 
\[
\eta_n = 
\left[ 
\sqrt{N} \left( w_k^* Q(\rho_k) w_{t/2+k} \right)_{1\le k\le t/2},  
\sqrt{n} \left( w_k^* X \widetilde Q(\rho_k) \tilde w_k  
\right)_{1\le k\le t/2}   
\right]^T  
\] 
converges in distribution towards $\mathcal{CN}(0, R)$ with 
\[
R = \begin{bmatrix} 
\diag\left( m'(\rho_k) - m(\rho_k)^2 \right)_{k=1}^{t/2} &  0\\
0& 
\diag\left( m(\rho_k) + \rho_k m'(\rho_k) \right)_{k=1}^{t/2} 
\end{bmatrix} .
\]
%where $0_{t/2 \times t/2}$ is a $t/2\times t/2$ null matrix. 
\end{proposition} 
Writing $Q - m I = (Q - \alpha I) + (\alpha - m)I$, and similarly for 
$\widetilde Q$, we obtain: 
\begin{corollary}
\label{clt-qf} 
Assume in addition that Assumption \ref{ass:strong-conv} is satisfied. Then 
\[ 
\xi_n = \sqrt{n} \left( 
W^* \Bigl( Q(\rho) - m(\rho) I_N \Bigr)W , \ 
\widetilde W^* \Bigl( \widetilde Q(\rho)  
- \tilde m(\rho) I_n \Bigr) \widetilde W , \ 
W^* X \widetilde Q(\rho) \widetilde W \right) 
\] 
is tight. 
\end{corollary}

Intuitively, tightness of $\xi_n$ leads to the tightness of the 
$\sqrt{n}(\hat\lambda_{k,n} - \rho_{i(k)})$. This is formalized by the
following proposition, proven in \ref{prf-tight}: 
\begin{proposition}
\label{spikes-tight} 
Assume the setting of Theorem \ref{2nd-order}. Then 
the sequences $\sqrt{n}(\hat\lambda_{k,n} - 
\rho_{i(k)})$ are tight for $1\le k\le r$. 
\end{proposition} 

The following lemma is proven in \ref{prf-proj}.  
\begin{lemma}
\label{proj-b}
Let Assumptions {\bf A5} and {\bf A6} hold true. Then the following 
convergences hold true: 
\begin{eqnarray*}
B^* B &\xrightarrow[n\to \infty]{}& I_r\ , \\
\frac 1{n^2}  B^* B'' &\xrightarrow[n\to \infty]{}& 
- \left(\frac{c^2 D^2}3\right) I_r\ ,  \\
  (B^\perp)^* B^\perp &\xrightarrow[n\to \infty]{}& I_r\ , \\
(B^\perp)^*  B &\xrightarrow[n\to \infty]{}& 0 \ , \\
% B_i^* \Pi_i B_i &\xrightarrow[n\to \infty]{}& I_{j_i}\ . \\
\| \Pi_i  - \Pi_{B_i} \| &\xrightarrow[n\to \infty]{}& 0 \ \text{for all } 
i = 1,\ldots, s  \\
\end{eqnarray*}
where $\Pi_{B_i}$ is the orthogonal projection matrix on the column space of
$B_i$. 
\end{lemma}

We now enter the proof of Theorem \ref{2nd-order}. 

Recall the definitions \eqref{eq:def-chi-hat} and \eqref{eq:def-zeta} of $\hat\chi$ and $\zeta$. 
In most of the proof, we shall focus on $\sqrt{n}(\hat\varphi_{1,n} 
-\varphi_1)$. Recalling that $\hat\chi'(\hat\varphi_1) = 0$ 
% and $\chi'(\varphi_1) = 0$, 
and performing a Taylor-Lagrange expansion of $\hat\chi'$ around $\varphi_1$, we 
obtain
\[
0 = \hat\chi'(\hat\varphi_1) = 
\hat\chi'(\varphi_1) + (\hat\varphi_1 - \varphi_1) \hat\chi''(\varphi_1) +
\frac{(\hat\varphi_1 - \varphi_1)^2}{2} \hat\chi^{(3)}(\bar\varphi_1) \ ,
\]
where $\hat\chi^{(3)}$ is the third derivative of $\hat\chi$ and where 
$\bar\varphi_1 \in [ \varphi_1 \wedge \hat\varphi_1, 
\varphi_1 \vee \hat\varphi_1 ]$. Hence 
\[
n^{3/2} (\hat\varphi_1 - \varphi_1) = 
- \frac{n^{-1/2}\hat\chi'(\varphi_1)}{n^{-2}\hat\chi''(\varphi_1) + 
0.5 n^{-2} (\hat\varphi_1 - \varphi_1) \hat\chi^{(3)}(\bar\varphi_1)}  \ .
\]

We start by characterizing the asymptotic behavior of the denominator of 
this equation: 
\begin{lemma} Assume that the setting of Theorem \ref{2nd-order} holds true. 
Then,
\label{denom}
\[
\frac{\hat\chi''(\varphi_1)}{n^2} + 
(\hat\varphi_1 - \varphi_1) \frac{\hat\chi^{(3)}(\bar\varphi_1)}{2n^2}   
\xrightarrow[n\to\infty]{\text{a.s.}}  
- \frac{c^2 D^2}{6} \ .
\]
\end{lemma}
\begin{proof}
We have
\begin{eqnarray} 
\frac{\hat\chi''(\varphi_1)}{n^2} &=& 
\frac{2}{n^2} \sum_{k=1}^r \zeta(\hat\lambda_k) 
| (b'_1)^* \widehat u_k |^2 + 
\frac{2}{n^2} \sum_{k=1}^r 
\Re \left( \zeta(\hat\lambda_k) b_1^* \widehat u_k  
\widehat u_k^* b''_1 \right) \ ,\nonumber \\
\frac{\chi''(\varphi_1)}{n^2} &=& 
\frac{2}{n^2} (b'_1)^* U U^* b'_1 + 
\frac{2}{n^2} \Re \left( b_1^* U U^* b''_1 \right) \ .
\label{chi''} 
\end{eqnarray} 
Theorem \ref{cvg-spk} along with the continuity of $\zeta$ on
$(\lambda_+,\infty)$, and Theorem \ref{bilin-forms} 
% with $(b_1, b_2) = (n^{-1}b'(\varphi_1), n^{-1}b'(\varphi_1))$ and
% $(b_1, b_2) = (b(\varphi_1), n^{-2}b''(\varphi_1))$ 
show that 
$$
\frac 1{n^2} \hat\chi''(\varphi_1) - \frac 1{n^2} \chi''(\varphi_1) \xrightarrow[n \to \infty]{a.s.} 0\ .
$$
Writing 
\[
\frac{1}{n^{2}} \chi''(\varphi_1) = \frac{2}{n^{2}}  
\sum_{i=1}^s \left( (b'_1)^* \Pi_i b'_1 + 
\Re( b_1^* \Pi_i b''_1 ) \right) , 
\] 
we have 
\begin{eqnarray*}
\frac 1{n^2} (b'_1)^* \Pi_i b'_1 &=& 
\left( -\frac{\imath cD}2 b_1 + 
\frac{cD}{2\sqrt{3}} b_1^\perp \right)^*  
\Pi_i 
\left( -\frac{\imath cD}2 b_1 + \frac{cD}{2\sqrt{3}} b_1^\perp \right)\ , \\
&\xrightarrow[n\to \infty]{}& \frac{c^2 D^2}4 \delta_{i,0} 
\end{eqnarray*}
by the first, fourth and fifth assertions of Lemma \ref{proj-b}. By the same 
lemma, 
\begin{eqnarray*}
\frac 1{n^2} b_1^* \Pi_i b''_1 -  
 \frac {\delta_{i,0}}{n^2} b_1^* b''_1 \ \to\ 0\quad \textrm{and}\quad 
\frac 1{n^2}b_1^* b''_1 \ \to\  -\frac{c^2 D^2}3\ .
\end{eqnarray*} 
Hence 
$n^{-2} \hat\chi''(\varphi_1) \to - c^2 D^2 / 6$. 

Furthermore, it is easily seen that $n^{-3}
\hat\chi^{(3)}(\bar\varphi_1)$ is bounded.  Since $n(\hat\varphi_1 -
\varphi_1) \toas 0$ by Theorem \ref{n-consist}, $n^{-2} (\hat\varphi_1
- \varphi_1) \hat\chi^{(3)}(\bar\varphi_1) \toas 0$, which establishes
the result.
\end{proof}

We now turn to the numerator $n^{-1/2} \hat\chi'(\varphi_1) = 
2 n^{-1/2} \sum_{k=1}^r 
\zeta(\hat\lambda_k) \Re\left( b_1^* \hat u_k \hat u_k^* b'_1 \right)$, and 
start with the following lemma: 
\begin{lemma} \label{num-chi}
Assume that the setting of Theorem \ref{2nd-order} holds true. Then
$$
\frac 1{\sqrt{n}} \hat\chi'(\varphi_1) - 2 \Re(\xi) \toP 0\ ,
$$ 
where 
\begin{equation} 
\xi = \sum_{i=1}^s 
\frac{\zeta(\rho_i)}{\imath \pi\sqrt{n}}
\oint_{\gamma_i} 
\left( \hat a^*(z, \varphi_1) \widehat H(z)^{-1} \hat a'_\varphi(z, \varphi_1) 
-  
a^*(z, \varphi_1) H(z)^{-1} a'_\varphi(z, \varphi_1) \right) dz , 
\label{eq-xi} 
\end{equation} 
and where the deterministic circle $\gamma_i$ encloses $\rho_i^{1/2}$ only and: 
\begin{eqnarray*} 
\hat a'_\varphi(z, \varphi) 
&=& \frac{\partial \hat a(z, \varphi)}{\partial \varphi} = 
\begin{bmatrix}
z U^* Q(z^2) \\ 
\Omega V^* \widetilde Q(z^2) X^* \end{bmatrix} b'(\varphi)\ ,\\ 
a'_\varphi(z, \varphi) &=& \frac{\partial a(z, \varphi)}{\partial \varphi} = 
\begin{bmatrix} z m(z^2) U^* \\ 0 \end{bmatrix} b'(\varphi) \ . 
\end{eqnarray*} 
\end{lemma} 
\begin{proof}
Recall the definition of $\hat\chi$ as given in \eqref{eq:def-chi-hat}. A direct computation yields:
\begin{eqnarray*}
  \hat \chi ' (\varphi) &=& 2 \sum_{k=1}^r \zeta(\hat\lambda_{k,n}) \Re \left(  b_1^*(\varphi) \hat u_k \hat u_k^* 
b'_1(\varphi) 
  \right)\ ,\\
&=& 2 \sum_{i=1}^s \sum_{k:i(k)=i} \zeta(\hat\lambda_{k,n}) \Re \left(  b_1^*(\varphi) \hat u_k \hat u_k^* 
b'_1(\varphi) 
  \right)\ .
\end{eqnarray*}
Recall that $r$ and $s$ are fixed and independent from $n$ by
\ref{ass:P-rough}. We start by showing that
\begin{equation}
\label{approx:chi'} 
\frac 1{\sqrt{n}} \hat\chi'(\varphi_1) -  
\frac 2{\sqrt{n}} \sum_{i=1}^s 
\zeta(\rho_i) 
\Re\left( b_1^* \widehat\Pi_i b'_1 \right)
\xrightarrow[n\to\infty]{{\mathcal P}} 0 .
\end{equation} 
Since $\sqrt{n} ( \zeta(\hat\lambda_{k,n}) - \zeta(\rho_{i(k)}))$ is
tight as a corollary of Proposition \ref{spikes-tight}, it will be enough to prove that 
$n^{-1} \Re\left( b_1^* \hat u_k \hat u_k^* b'_1 \right) 
\to 0$ in probability for every $k$. By the definition \eqref{eq:def-B-prime} of $B^\perp$, we have
$$
\frac 1n\,  \Re\left( b_1^* \hat u_k \hat u_k^* b'_1 \right) 
= 
\frac{cD}{2\sqrt{3}}\, \Re\left( b_1^* \hat u_k \hat u_k^* b_1^\perp \right)\ .
$$ 
By Cauchy-Schwarz inequality,   
\[
\left| b_1^* \hat u_k \hat u_k^* b_1^\perp \right|^2 
\leq 
 b_1^* \widehat\Pi_{i(k)} b_1 \ 
 (b_1^\perp)^* \widehat\Pi_{i(k)} b_1^\perp .  
\]
By Theorem \ref{bilin-forms}, 
\[
 b_1^* \widehat\Pi_{i(k)} b_1 \ 
 (b_1^\perp)^* \widehat\Pi_{i(k)} b_1^\perp 
\ - \ 
\zeta(\rho_{i(k)})^{-2} \ b_1^* \Pi_{i(k)} b_1 \  
(b_1^\perp)^* \Pi_{i(k)} b_1^\perp  \ \toas \ 0 , 
\] 
and by Lemma \ref{proj-b}, 
$b_1^* \Pi_{i(k)} b_1 \  (b_1^\perp)^* \Pi_{i(k)} b_1^\perp \ \to \ 0$ 
(consider alternatively the cases $i(k)=1$ and $i(k) > 1$) 
which proves \eqref{approx:chi'}.  

Now, applying \eqref{eq-expansion-aH} and \eqref{eq-a-hat-a}, 
and taking up an argument used in the proof of Theorem \ref{bilin-forms}, we
have  
\begin{align*} 
2 \sum_{i=1}^s 
\frac{\zeta(\rho_i)}{\sqrt{n}} 
\Re\left( b_1^* \widehat\Pi_i b'_1 \right)
&=  
2 \sum_{i=1}^s 
\Re\left( \frac{-\zeta(\rho_i)}{\imath \pi \sqrt{n}}
\oint_{{\mathcal C}_i} \begin{bmatrix} b^*_1 & 0 \end{bmatrix}
{\bf Q}(z) \begin{bmatrix} b'_1 \\ 0 \end{bmatrix} \ dz \right) \\
&\phantom{=} + 
2 \sum_{i=1}^s 
\Re \left( \frac{\zeta(\rho_i)}{\imath \pi\sqrt{n}}
\oint_{{\mathcal C}_i} 
\hat a^*(z, \varphi_1) \widehat H(z)^{-1} \hat a'_\varphi(z, \varphi_1) \, dz  
\right) \\ 
&= 
2 \sum_{i=1}^s 
\Re \left( \frac{\zeta(\rho_i)}{\imath \pi\sqrt{n}}
\oint_{\gamma_i} 
\hat a^*(z, \varphi_1) \widehat H(z)^{-1} \hat a'_\varphi(z, \varphi_1) \, dz  
\right) 
\end{align*} 
with probability one for $n$ large. 
On the other hand, recalling \eqref{eq-T_k}, we have 
\[ 
0 = \chi'(\varphi_1) =  2 \sum_{i=1}^s 
\Re\left( \frac{\zeta(\rho_i)}{\imath \pi} 
\oint_{\gamma_i} 
a^*(z, \varphi_1) H(z)^{-1} a'_\varphi(z, \varphi_1) \, dz  
\right) \ ,
\]
which proves the result. 
\end{proof} 

Write $\widehat H(z) = H(z) + E(z)$ and $\hat a(z, \varphi) = 
a(z, \varphi) + e(z, \varphi)$. To be more specific, 
\begin{equation}
\label{E(z)} 
E({z}) = \begin{bmatrix}
{z} U^* (Q(z^2) - m(z^2)I_N) U & U^* X \widetilde Q(z^2) V \Omega \\
 \Omega V^* \widetilde Q(z^2) X^* U & 
{z} \Omega V^* (\widetilde Q(z^2) - \tilde m(z^2)I_n) V \Omega 
\end{bmatrix}  
\end{equation} 
and 
\[
e(z, \varphi) = \begin{bmatrix}
z U^* \left( Q(z^2) - m(z^2)I\right) \\
\Omega V^* \widetilde Q(z^2) X^*  
\end{bmatrix} b(\varphi) . 
\] 
Write $e'_\varphi(z, \varphi) = \partial e(z,\varphi) /\partial \varphi$. 
For a given $z \in \gamma_i$, $\widehat H^{-1} = 
H^{-1} - H^{-1} E H^{-1} + {\mathcal O}(\| E \|^2)$. This suggests the 
following development 
\begin{align*}
\xi &= 
\sum_{i=1}^s \left( 
\frac{\zeta(\rho_i)}{\imath \pi\sqrt{n}}
\oint_{\gamma_i} 
a^*(z, \varphi_1) H(z)^{-1} e'_\varphi(z, \varphi_1) \, dz \right. \\
& \phantom{\sum_i^s} \quad + 
\frac{\zeta(\rho_i)}{\imath \pi\sqrt{n}}
\oint_{\gamma_i} 
e^*(z, \varphi_1) H(z)^{-1} a'_\varphi(z, \varphi_1) \, dz  \\
& \phantom{\sum_i^s} \quad - \left. 
\frac{\zeta(\rho_i)}{\imath \pi\sqrt{n}}
\oint_{\gamma_i} 
a^*(z, \varphi_1) H(z)^{-1} E(z) H(z)^{-1} a'_\varphi(z, \varphi_1) \, dz 
\ + q_i \right) \\
&= \sum_{i=1}^s ( X_{1,i} + X_{2,i} + X_{3,i} + q_i ) \ . 
\end{align*} 
where the terms $q_i$ are ``higher order terms'' that appear when we expand
the r.h.s. of \eqref{eq-xi}. 
We first handle the terms $X_{k,i}$'s, then $q_i$.

\subsection*{The terms $X_{1,i}$}
Writing $U_n = \begin{bmatrix} U_{1,n} \cdots U_{s,n} \end{bmatrix}$ and 
$V_n = \begin{bmatrix} V_{1,n} \cdots V_{s,n} \end{bmatrix}$ where both 
$U_{i,n}$ and $V_{i,n}$ have $j_i$ columns, and recalling \eqref{inv(H)}, 
we have   
\begin{align*}
X_{1,i} &= 
\frac{\zeta(\rho_i)}{\imath \pi\sqrt{n}} \sum_{\ell=1}^s 
\oint_{\gamma_i} 
\begin{bmatrix} z m(z^2) b_1^* U_\ell & 0 \end{bmatrix} 
\times \frac{
\begin{bmatrix} z \tilde m(z^2) \omega_\ell^2 & -1 \\
- 1 & z m(z^2) \end{bmatrix} 
\otimes I_{j_\ell} 
}{z^2 m(z^2) \tilde m(z^2) \omega_\ell^2 - 1} \times \\
& 
\ \ \ \ \ \ \ \ \ \ \ \ \ \ \ \ \ \ \ \ \ \ \ \ \ \ \ \ \ \ \ \ 
\ \ \ \ \ \ \ \ \ \ \ \ \ \ \ \ \ \ \ \ \ \ \ \ \ \ \ \ \ \ \ \ 
\begin{bmatrix}
z U_\ell^* \left( Q(z^2) - m(z^2) I \right) b_1' \\
\omega_\ell V_\ell^* \widetilde Q(z^2) X^* b_1' 
\end{bmatrix} \, dz \\ 
&= 
\frac{\zeta(\rho_i)}{\imath \pi\sqrt{n}} \sum_{\ell=1}^s 
\oint_{\gamma_i} 
\frac{z^3 \omega_\ell^2 m(z^2) \tilde m(z^2) b_1^* \Pi_\ell 
\left( Q(z^2) - m(z^2) I \right) b_1'}
{z^2 m(z^2) \tilde m(z^2) \omega_\ell^2 - 1} \, dz \\
& \ \ 
- 
\frac{\zeta(\rho_i)}{\imath \pi\sqrt{n}} \sum_{\ell=1}^s \oint_{\gamma_i} 
\frac{\omega_\ell z m(z^2) b_1^* U_\ell V_\ell^* \widetilde Q(z^2) X^* b_1'}
{z^2 m(z^2) \tilde m(z^2) \omega_\ell^2 - 1} \, dz \\
&= 
\frac{\zeta(\rho_i)}{2 \imath \pi\sqrt{n}} \sum_{\ell=1}^s 
\oint_{\gamma'_i} 
\frac{w \omega_\ell^2 m(w) \tilde m(w) b_1^* \Pi_\ell 
\left( Q(w) - m(w) I \right) b_1'}
{w m(w) \tilde m(w) \omega_\ell^2 - 1} \, dw \\
& \ \ 
- 
\frac{\zeta(\rho_i)}{2 \imath \pi\sqrt{n}} \sum_{\ell=1}^s 
\oint_{\gamma'_i} 
\frac{\omega_\ell m(w) b_1^* U_\ell V_\ell^* \widetilde Q(w) X^* b_1'}
{w m(w) \tilde m(w) \omega_\ell^2 - 1} \, dw 
\end{align*} 
where $\gamma'_i$ encloses $\rho_i$ only.
These integrals are zero for $\ell \neq i$. For large $n$ and with probability
one, none of the numerators has a pole within $\gamma'_i$, hence by the
Residue Theorem
\[
X_{1,i} =  
\frac{b_1^* \Pi_i \left( Q(\rho_i) - m(\rho_i) I \right) b_1'}
{\sqrt{n} m(\rho_i)} 
- \frac{\omega_i b_1^* U_i V_i^* \widetilde Q(\rho_i) X^* b_1'}
{\sqrt{n}}% \quad \text{w.p. } 1 \ \text{for large } n .
\]
a.s. for $n$ large enough.

Due to the bounded character of $\| n^{-1} b' \|$ and to Corollary 
\ref{clt-qf}, $X_{1,i}$ is tight for every $i$. By Lemma \ref{proj-b}, 
\[
X_{1,i} \asymp \delta_{i-1,0} \left( 
\frac{b_1^* \left( Q(\rho_1) - m(\rho_1) I \right) b_1'}
{\sqrt{n} m(\rho_1)} 
- \frac{\omega_1 b_1^* U_1 V_1^* \widetilde Q(\rho_1) X^* b_1'}
{\sqrt{n}} \right) . 
\]

\subsection*{The terms $X_{2,i}$}
We have here
\begin{align*} 
X_{2,i} &= 
\frac{\zeta(\rho_i)}{\imath \pi\sqrt{n}} \sum_{\ell=1}^s 
\oint_{\gamma_i} 
\begin{bmatrix} z b_1^* \left( Q(z^2) - m(z^2) I \right) U_\ell & 
\omega_\ell b_1^* X \widetilde Q(z^2) V_\ell \end{bmatrix} \\
& 
\ \ \ \ \ \ \ \ \ \ \ \ \ \ \ \ \ \ \ \ \ \ \ \ \ \ \ \ 
\times 
\frac{
\begin{bmatrix} z \tilde m(z^2) \omega_\ell^2 & -1 \\
- 1 & z m(z^2) \end{bmatrix} \otimes I_{j_\ell} }
{z^2 m(z^2) \tilde m(z^2) \omega_\ell^2 - 1} 
\times 
\begin{bmatrix}
z m(z^2) U_\ell^* b_1' \\
0
\end{bmatrix} \, dz \\ 
&= 
\frac{\zeta(\rho_i)}{2\imath \pi\sqrt{n}} 
\sum_{\ell=1}^s \oint_{\gamma'_i} 
\frac{w m(w) \tilde m(w) \omega_\ell^2 
b_1^* \left( Q(w) - m(w) I \right) \Pi_\ell b_1'}
{w m(w) \tilde m(w) \omega_\ell^2 - 1} \, dw \\
& \ \ \ - 
\frac{\zeta(\rho_i)}{2\imath \pi\sqrt{n}} 
\sum_{\ell=1}^s 
\oint_{\gamma'_i} 
\frac{\omega_\ell m(w) b^* X \widetilde Q(w) V_\ell U_\ell^* b_1'}
{w m(w) \tilde m(w) \omega_\ell^2 - 1} \, dw \\
&= 
\frac{b_1^* \left( Q(\rho_i) - m(\rho_i) I \right) \Pi_i b_1'}
{\sqrt{n} m(\rho_i)} 
- \frac{\omega_i b_1^* X \widetilde Q(\rho_i) V_i U_i^* b_1'}{\sqrt{n}}
\quad \text{w.p.} \ 1 \ \text{for large} \ n  \\ 
&\asymp \delta_{i-1,0} \left( 
\frac{b_1^* \left( Q(\rho_i) - m(\rho_i) I \right) \Pi_1 b_1'}
{\sqrt{n} m(\rho_i)} 
- \frac{\omega_i b_1^* X \widetilde Q(\rho_i) V_1 U_1^* b_1'}{\sqrt{n}}
\right)  
\end{align*} 
by Corollary \ref{clt-qf} and Lemma \ref{proj-b}.

\subsection*{The terms $X_{3,i}$}

From \eqref{inv(H)} and \eqref{E(z)}, we have 
\begin{align*}
X_{3,i} &= 
-\frac{\zeta(\rho_i)}{\imath \pi\sqrt{n}}
\oint_{\gamma_i} 
\sum_{p, \ell=1}^s \begin{bmatrix} z m(z^2) b_1^* U & 0 \end{bmatrix} \\
& \ \ \ \ \ \ 
\times \frac{
\left( \begin{bmatrix} z \tilde m(z^2) \omega_p^2 & - 1 \\ 
- 1 & z m(z^2) \end{bmatrix} 
\otimes {\mathcal I}_p \right) 
E(z) 
\left( \begin{bmatrix} z \tilde m(z^2) \omega_\ell^2 & - 1 \\ 
- 1 & z m(z^2) \end{bmatrix} 
\otimes {\mathcal I}_\ell \right) 
}
{(z^2 m(z^2) \tilde m(z^2) \omega_p^2 -1)
(z^2 m(z^2) \tilde m(z^2) \omega_\ell^2 -1)} \times \\ 
& 
\ \ \ \ \ \ \ \ \ \ \ \ \ \ \ \ \ \ \ \ \ \ \ \ 
\ \ \ \ \ \ \ \ \ \ \ \ \ \ \ \ \ \ \ \ \ \ \ \ 
\ \ \ \ \ \ \ \ \ \ \ \ \ \ \ \ \ \ \ \ \ \ \ \ 
\begin{bmatrix} z m(z^2) U^* b_1' \\ 0 \end{bmatrix} 
\, dz \\
&= 
-\frac{\zeta(\rho_i)}{\imath \pi\sqrt{n}}
\oint_{\gamma_i} 
\sum_{p, \ell=1}^s 
\begin{bmatrix} \omega_p^2 z^2 m(z^2)\tilde m(z^2) & - zm(z^2) \end{bmatrix} 
\times \\
& 
\frac{
\begin{bmatrix} 
z b_1^* \Pi_p (Q(z^2) - m(z^2)I) \Pi_\ell b_1' & 
\omega_\ell b_1^* \Pi_p X \widetilde Q(z^2) V_\ell U_\ell^* b_1' \\
\omega_p b_1^* U_p V_p^* \widetilde Q(z^2) X^* \Pi_\ell b_1' & 
z \omega_p \omega_\ell b_1^* U_p V_p^* (\widetilde Q(z^2) - \tilde m(z^2)) 
V_\ell U_\ell^* b_1' 
\end{bmatrix}
}
{(z^2 m(z^2) \tilde m(z^2) \omega_p^2 -1)
(z^2 m(z^2) \tilde m(z^2) \omega_\ell^2 -1)} \\ 
& 
\ \ \ \ \ \ \ \ \ \ \ \ \ \ \ \ \ \ \ \ \ \ \ \ 
\ \ \ \ \ \ \ \ \ \ \ \ \ \ \ \ \ \ \ \ \ \ \ \ 
\ \ \ \ \ \ \ \ \ \ \ \ \ \ \ \ \ \ 
\times \begin{bmatrix} \omega_\ell^2 z^2 m(z^2) \tilde m(z^2) \\ 
- z m(z^2) 
\end{bmatrix} \, dz \\
&= 
- \frac{\zeta(\rho_i)}{2 \imath \pi}
\oint_{\gamma_i'} \sum_{p, \ell=1}^s 
\frac{G_{p,\ell}(w)} 
{(w m(w) \tilde m(w) \omega_p^2 -1)
(w m(w) \tilde m(w) \omega_\ell^2 -1)} \, dw 
\end{align*} 
where 
\begin{multline*} 
G_{p\ell}(w) = n^{-1/2}\left( 
\omega_p^2 \omega_\ell^2 w^2 m(w)^2 \tilde m(w)^2 \, 
b_1^* \Pi_p (Q(w) - m(w)I) \Pi_\ell b_1' \phantom{\widetilde Q} \right. \\ 
- \omega_p \omega_\ell^2 w m(w)^2 \tilde m(w) \, 
b_1^* U_p V_p^* \widetilde Q(w) X^* \Pi_\ell b_1' \\ 
- \omega_p^2 \omega_\ell w m(w)^2 \tilde m(w) \, 
b_1^* \Pi_p X \widetilde Q(w) V_\ell U_\ell^* b_1' \\  
\left. + \omega_p \omega_\ell w m(w)^2 \, 
b_1^* U_p V_p^* (\widetilde Q(w) - \tilde m(w)I) 
V_\ell U_\ell^* b_1' \right) . 
\end{multline*} 
For large $n$ and with probability one, the $G_{p\ell}(w)$ are holomorphic
functions in a domain enclosing $\gamma'_i$, and $G_{p\ell}(w)$ does 
not cancel any of the terms of the denominator. 
The integrals of all terms in the sum such that $p \neq i$ and $\ell \neq i$ 
are zero. Each of the integrands of the terms 
$p = i, \ell\neq i$ or $p \neq i, \ell=i$ has a pole with degree one, and 
the corresponding integrals are of the form $K_{i\ell} G_{i\ell}(\rho_i)$ 
or $K_{pi} G_{pi}(\rho_i)$ where the $K_{i\ell}$ and $K_{pi}$ are real
constants. By inspecting the expression of $G_{p\ell}$ and by using 
Corollary \ref{clt-qf} and Lemma \ref{proj-b}, it can be seen 
that these terms converge to zero in probability. It remains to study the 
term $p=\ell=i$, which has a degree $2$ pole. Recalling that the residue of 
a meromorphic function $f(z)$ that has a pole with degree $2$ at $z_0$ is 
$\lim_{z\to z_0} d\left( (z-z_0)^2 f(z) \right)/ dz$ 
and letting $g_\ell(z) = z m(z) \tilde m(z) \omega_\ell^2 -1$, the integral
of this term is 
\[
\zeta(\rho_i) 
\left( 
\frac{G_{ii}(\rho_i) g_i''(\rho_i)}{g_i'(\rho_i)^3}  
- \frac{G'_{ii}(\rho_i)}{g_i'(\rho_i)^2} 
\right) . 
\] 
Thanks to Corollary \ref{clt-qf} and Lemma \ref{proj-b}, 
$\Re(G_{ii}(\rho_i)) \toP 0$. The same can be said
about $G'_{ii}(\rho_i)$ after a simple modification of Proposition 
\ref{clt-haar} and Corollary \ref{clt-qf}. In conclusion, 
\[
\forall i = 1,\ldots, s, \quad \Re(X_{3,i}) \toP 0  . 
\] 

\subsection*{The terms $q_i$}
These are the higher order terms that appear when we expand the right 
hand side of \eqref{eq-xi}. We shall work here on one of these terms, 
namely 
\[
\varepsilon = 
\frac{\zeta(\rho_i)}{\imath \pi\sqrt{n}}
\oint_{\gamma_i} 
a^*(z, \varphi_1) 
\left( \widehat H(z)^{-1} - H(z)^{-1} + H(z)^{-1} E(z) H(z)^{-1} \right)
a'_\varphi(z, \varphi_1) \, dz 
\]
and show that $\varepsilon \toP 0$. The other higher order terms can be handled
similarly. Writing $z = \sqrt{\rho_i} + R \exp(2 \imath\pi \theta)$ on the
circle $\gamma_i$, we have 
\[
|\varepsilon| \leq K \sqrt{n} \int_0^1 
\| \widehat H(z)^{-1} - H(z)^{-1} + H(z)^{-1} E(z) H(z)^{-1} \| \, d\theta 
\]
where $K$ is a constant whose value can change from line to line, but
which remains independent from $n$. Let $\phi$ be a function from
$[0,1]$ to a normed vector space. If $\phi$ is twice differentiable on
$(0,1)$, then it is known that $\| \phi(1) - \phi(0) - \phi'(0) \|
\leq \sup_{t \in (0,1)} 0.5 \| \phi''(t) \|$.  

Setting $\phi(t) = ( H
+ t E )^{-1}$ and recalling that $\hat H = H + E$, we have $\phi(1) = \hat H$, 
$\phi(0) = H$ and $\phi''(t) = (H+tE)^{-1}
E(H+tE)^{-1}E(H+tE)^{-1}$, hence
\[
\| \widehat H(z)^{-1} - H(z)^{-1} + H(z)^{-1} E(z) H(z)^{-1} \| \leq 
K \| E(z) \|^2 
\]
for $z \in \gamma_i$. Write $Q - mI = (Q-\alpha I) + (\alpha-m)I$ and $\widetilde Q - \tilde m I= 
(\widetilde Q - \tilde \alpha I) + (\tilde \alpha - \tilde m)I$, and decompose $E$ as defined in \eqref{E(z)}
as $E= E_1 + E_2$ where
\begin{eqnarray*}
E_1({z}) &=& \begin{bmatrix}
{z} U^* (Q(z^2) - \alpha(z^2)I_N) U & U^* X \widetilde Q(z^2) V \Omega \\
 \Omega V^* \widetilde Q(z^2) X^* U & 
{z} \Omega V^* (\widetilde Q(z^2) - \tilde \alpha (z^2)I_n) V \Omega 
\end{bmatrix}\ ,  \\
E_2({z}) &=& \begin{bmatrix}
{z} U^* (\alpha(z^2) - m(z^2))I_N) U & 0 \\
 0 & {z} \Omega V^* (\widetilde \alpha(z^2) - \tilde m(z^2))I_n V \Omega 
\end{bmatrix}  \ .
\end{eqnarray*}
Consider any element of $E_1$, for instance 
$z u_1^* (Q(z^2) - \alpha(z^2)I) u_1$. By Lemma \ref{lm-fq-haar}, 
\[
\sqrt{n} \EE\left( \int_0^1 
\1_{{\mathcal O}_n} \left| u_1^* (Q - \alpha ) u_1 \right|^2 d\theta \right) 
= 
\sqrt{n} \int_0^1 
\EE \1_{{\mathcal O}_n} \left| u_1^* (Q - \alpha ) u_1 \right|^2 d\theta 
\leq \frac{K}{\sqrt{n}} 
\] 
which shows that $\sqrt{n} \int_0^1 \| E_1 \|^2 d\theta \toP 0$. 
 
We now prove that $\sqrt{n} \int_0^1 \| E_2 \|^2 d\theta \toP 0$.  In
the space of probability measures on $\RR$ endowed with the weak
convergence metric, in order to prove that a sequence converges weakly
to $\mu$, it is enough to prove that from any sequence, we can extract
a subsequence along which the weak convergence to $\mu$ holds true.
We shall show along this principle that $\sqrt{n} \int_0^1 \| E_2 \|^2
d\theta \toP 0$.  Consider the term $\sqrt{n} (\alpha - m)$.  Let
$(z_k)$ be a denumerable sequence of points in $\CC - [0, \lambda_+]$
with an accumulation point in that set. By \ref{ass:strong-conv}, from
every sequence, there is subsequence $n_\ell$ such that $\sqrt{n_\ell}
(\alpha_{n_\ell}(z_1) - m(z_1)) \to 0$ almost surely (recall that the
convergence in probability implies the a.s. convergence along a
subsequence). By Cantor's diagonal argument, we can extract a subsequence
(call it again $n_\ell$) such that $\sqrt{n_\ell}
(\alpha_{n_\ell}(z_k) - m(z_k)) \to 0$ almost surely for every $k$. By
the normal family theorem, there is a subsequence along which the function
$\sqrt{n_\ell} (\alpha_{n_\ell} - m) \to 0$ uniformly on $\gamma_i$ a.s.
Repeating the argument for $\sqrt{n}
(\tilde\alpha - \tilde m)$, there is a subsequence $n_\ell$ along
which $\sqrt{n_\ell} \int_0^1 \| E_2 \|^2 d\theta \toas 0$, hence weakly.
Necessarily, $\sqrt{n_\ell} \int_0^1 \| E_2 \|^2 d\theta$ converges weakly to zero.
Now since the weak convergence to a constant is equivalent to the convergence 
in probability to the same constant, we obtain the desired result. We have 
finally shown that:
\[
\forall i = 1,\ldots, s, \quad q_i \toP 0\  . 
\]

\subsection*{Final derivations} 
Write $\hat{\bs\chi}' = \begin{bmatrix} \hat\chi'(\varphi_1), \ldots, 
\hat\chi'(\varphi_r) \end{bmatrix}$. Generalizing the previous argument to all
the $\varphi_k$ and gathering the results, we obtain   
\begin{align*} 
n^{-1/2} \hat{\bs\chi'}  &\asymp 
2 \Re\left[ 
\frac{b_k^* \left( Q(\rho_{i(k)}) - m(\rho_{i(k)}) I \right) b_k'}
{\sqrt{n} m(\rho_{i(k)})} 
+ 
\frac{b_k^* \left( Q(\rho_{i(k)}) - m(\rho_{i(k)}) I \right) \Pi_{i(k)} b_k'}
{\sqrt{n} m(\rho_{i(k)})} \right. \\
& \phantom{\asymp 2 \Re \bigl(} 
\left. - \frac{\omega_{i(k)} b_k^* U_{i(k)} V_{i(k)}^* 
\widetilde Q(\rho_{i(k)}) X^* b_k'}
{\sqrt{n}} 
- \frac{\omega_{i(k)} b_k^* X \widetilde Q(\rho_{i(k)}) V_{i(k)} U_{i(k)}^* b_k'}{\sqrt{n}}
\right]_{k=1}^r  \\
&\asymp \frac{cD}{\sqrt{3}} \sqrt{n} \Re\left[  
\frac{b_k^* Q(\rho_{i(k)}) b_k^\perp}
{m(\rho_{i(k)})} 
- 
\omega_{i(k)} b_k^* U_{i(k)} V_{i(k)}^* 
\widetilde Q(\rho_{i(k)}) X^* b_k^\perp \right]_{k=1}^r 
\end{align*} 
By Lemma \ref{proj-b}, matrix $A = \begin{bmatrix} V_{i(k)} U_{i(k)}^* b_k 
\end{bmatrix}_{k=1}^r$ satisfies $A^* A \to I_r$. Recall from the same 
lemma that $B^* B \to I_r$, $(B^\perp)^* B^\perp \to I_r$ and 
$(B^\perp)^* B \to 0$. Hence, Proposition \ref{clt-haar} can be applied to
the r.h.s. of this expression, and $n^{-1/2} \hat{\bs\chi'}$ 
converges in law to 
\[
\mathcal{N}\left( 0, 
\frac{c^2 D^2}{6} 
\diag\left(
\frac{m'(\rho_{i(k)}) - m(\rho_{i(k)})^2}{c m(\rho_{i(k)})^2}
+ \omega_{i(k)}^2 
\left( m(\rho_{i(k)}) + \rho_{i(k)} m'(\rho_{i(k)})\right) \right)_{k=1}^r 
\right) 
\]
It remains to recall Lemmas \ref{denom} and \ref{num-chi}
to terminate the proof of Theorem \ref{2nd-order}. 

\appendix 
\section{Proof of Proposition \ref{clt-haar}} 
\label{clt-fq}
The tightness of $\xi_n$ follows from Lemmas \ref{lm-fq-haar} and 
\ref{lm-fq-two-haars} with $p=2$ and from the application of 
Chebyshev's inequality. \\ 
Let $Z = [ z_{i,k} ]_{i,k=1}^{N,t}$ and 
$\widetilde Z = [ \tilde z_{i,k} ]_{i,k=1}^{n,t}$ be $N\times t$ and 
$n\times t$ standard Gaussian random matrices chosen such that $Z$, 
$\widetilde Z$ and the $N \times N$ matrix $\Gamma$ of singular values of $X$ 
are independent.
For $k=1,\ldots,t/2$, let  
$D_k = \diag( d_{i,k} )_{i=1}^N = (\Gamma^2 - \rho_k)^{-1}$  
and $C_k = \diag( c_{i,k} )_{i=1}^N = \Gamma (\Gamma^2 - \rho_k)^{-1}$.
Then 
\begin{multline*} 
\eta_n \stackrel{{\mathcal D}}{=} \left[
\sqrt{N} \left( 
\left[ (Z^* Z)^{-1/2} Z^* \left(D_k - \frac{\tr D_k}{N} \right) 
Z (Z^* Z)^{-1/2} \right]_{k, k+t/2} 
\right)_{k=1,\ldots,t/2},  \right. \\
\left. 
\sqrt{n} \left( 
\left[ (Z^* Z)^{-1/2} Z^* C_k \widetilde Z[1;N] 
(\widetilde Z^* \widetilde Z)^{-1/2} \right]_{k,k} 
\right)_{k=1,\ldots,t/2} 
\right]^T
\end{multline*}
where $\widetilde Z[1;N]$ is $\widetilde Z$ truncated to its first $N$ rows. 
By the Law of Large Numbers, $N^{-1} Z^* Z \to I_t$ and 
$n^{-1} \widetilde Z^* \widetilde Z \to I_t$ almost surely. Hence, if we 
show that the multidimensional random variables 
$A_{k,n} = N^{-1/2} Z^* (D_k - N^{-1} \tr D_k) Z$ and 
$B_{k,n} = N^{-1/2} Z^* C_k \widetilde Z[1;N]$ are tight for 
$k=1, \ldots,t/2$, and 
\begin{multline*} 
\bar\eta_n = \frac{1}{\sqrt{N}} \left[ \left( 
\left[ Z^* \left(D_k - \frac{\tr D_k}{N} \right) 
Z \right]_{k, k+t/2} 
\right)_{k=1,\ldots,t/2},  \right. \\
\left. 
\left( 
\left[ Z^* C_k \widetilde Z[1;N] 
\right]_{k,k} 
\right)_{k=1,\ldots,t/2} 
\right]^T
\end{multline*}
converges in law towards $\mathcal{CN}(0,R)$, the second result of Proposition 
\ref{clt-haar} is proven. From {\bf A3} and {\bf A4}, 
\begin{gather*} 
\frac 1N \sum_{i=1}^N \Bigl( d_{i,k} - \frac{\tr D_k}{N} \Bigr)^2 = 
\frac 1N \tr Q(\rho_k)^2 - \left( \frac 1N \tr Q(\rho_k) \right)^2 
\xrightarrow[n\to\infty]{\text{a.s.}} 
m'(\rho_k) - m(\rho_k)^2 ,  \ \text{and} \\
\frac 1N \sum_{i=1}^N c_{i,k}^2 = 
\frac 1N \tr Q(\rho_k) + \frac{\rho_k}{N} \tr Q(\rho_k)^2 
\xrightarrow[n\to\infty]{\text{a.s.}} 
m(\rho_k) + \rho_k m'(\rho_k)
\end{gather*} 
for all $k=1,\ldots,t/2$. Recalling that $Z$ and $\widetilde Z$ are 
standard Gaussian, it results that 
$\limsup_n \EE\left[ \| A_{k,n} \|^2 \, \| \, \Gamma_n \right]$ and 
$\limsup_n \EE\left[ \| B_{k,n} \|^2 \, \| \, \Gamma_n \right]$ are bounded
w.p.~1 by a constant. Tightness of the $A_{k,n}$ and $B_{k,n}$
follows. Now we have 
\begin{align*} 
\bar\eta_n &= 
\frac{1}{\sqrt{N}} 
\sum_{i=1}^N 
\left[ 
\left( (d_{i,k} - N^{-1} \tr D_k) 
z_{i,k}^* z_{i, k+t/2} \right)_{k=1,\ldots,t/2}, 
\left( c_{i,k} z_{i,k}^* \tilde z_{i, k} \right)_{k=1,\ldots,t/2} 
\right]^T \\
&=  \frac{1}{\sqrt{N}} \sum_{i=1}^N {\bf u}_{i,n} . 
\end{align*}
Observe that covariance matrix of $\bar\eta_n$ conditional to $\Gamma_n$
converges almost surely to $R$. 
Moreover, thanks to {\bf A4}, it is easy to see that the Lyapunov condition 
\[
\frac{1}{N^{1+a}} \sum_{i=1}^n \EE\left[ \| {\bf u}_{i,n} \|^{2(1+a)} 
\, \| \, \Gamma_n \right] 
\xrightarrow[n\to\infty]{\text{a.s.}} 0 
\]
is satisfied for any $a > 0$, hence $\bar\eta_n \toL 
\mathcal{CN}(0,R)$ which completes the proof of Proposition \ref{clt-haar}. 

\section{Sketch of the proof of Proposition \ref{spikes-tight}.}
\label{prf-tight} 
For $k=1,\dots, r$, let $\bar\rho_{k,n}$ be the solutions of the equation
$\omega_{k,n}^2 g(\rho) = 1$, where we recall that the $\omega_{k,n}^2$ are
the diagonal elements of matrix $\Omega_n$. 
Then, by a simple extension to the case $r\geq 1$ of the proof of 
\cite[Th. 2.15]{ben-rao-11}, one can show that the sequences 
$\sqrt{n} ( \hat\lambda_{k,n} - \bar\rho_{k,n})$ are tight. 
To obtain the result, we show that 
$\sqrt{n}( \bar\rho_{k,n} - \rho_{i(k)} ) = {\mathcal O}(1)$. 
Since $g$ is decreasing, this amounts to showing that 
$\sqrt{n}( \omega_{k,n}^2 - \omega_{i(k)}^2 ) = {\mathcal O}(1)$. 
Since the non zero eigenvalues of $P P^*$ coincide with
those of $B^* B \, S^* S$, it will be enough to prove that 
$\sqrt{n}( B^* B \, S^* S - O ) = {\mathcal O}(1)$. 
It is clear that $B^* B = I_r + n^{-1} A$ where $\sup_n \| A \| < \infty$, 
hence $\sqrt{n}( B^* B O - O ) \to 0$. By the last item in Assumption {\bf A6}, 
$\sqrt{n} B^* B (S^* S - O ) = {\mathcal O}(1)$, and the proposition is shown.

\section{Proof of Lemma \ref{proj-b}.}
\label{prf-proj} 
Observing that 
\[
b'(\varphi) = \frac{- \imath D}{\sqrt{N}} \begin{bmatrix} \ell 
\exp(-\imath D \ell \varphi) \end{bmatrix}_{\ell=0}^{N-1} \quad \text{and}
\quad  
b''(\varphi) = \frac{-D^2}{\sqrt{N}} \begin{bmatrix} \ell^2 
\exp(-\imath D \ell \varphi) \end{bmatrix}_{\ell=0}^{N-1}, 
\]
and using the fact that 
$N^{-(K+1)} \sum_{\ell=0}^{N-1} \ell^K \exp(\imath \alpha\ell) 
\to \delta_{\alpha,0} / (K+1)$ for $\alpha \in[-\pi, \pi]$, we have
$B^* B \to I_r$, $n^{-1} B^* B' \to -(\imath cD/2) I_r$, 
$n^{-2} (B')^* B' \to (c^2 D^2/3) I_r$, and 
$n^{-2} B^* B'' \to - (c^2 D^2/3) I_r$. \\ 
Writing $B^\perp = 2 \sqrt{3} (ncD)^{-1} B' + \imath \sqrt{3} B$ and
replacing in the above convergences, the stated properties of $B^\perp$ 
become straightforward. \\ 
We now show the last convergence. Assume without generality 
loss that $i=1$ and recall that $S^*S \to O^2$. 
Consider the isometry matrices 
$W = B (B^* B)^{-1/2}$ and $Z = S (S^* S)^{-1/2}$, and let 
$A = (B^*B)^{1/2} (S^*S)^{1/2}$, resulting in $P = W A Z^*$. 
Notice that the singular values of $A$ coincide with those of $P$ apart from
the zeros. 
Let ${\bs \pi}_1$ be the orthogonal projection matrix on the 
eigenspace of $AA^*$ associated with the eigenvalues 
$\omega_{1,n}^2, \ldots, \omega_{j_1,n}^2$. With these notations, 
$\Pi_1 = W {\bs \pi}_1 W^*$ and $\Pi_{B_1} = B_1 (B_1^* B_1)^{-1} B_1^*$. 
We have $A \to O$, hence ${\bs\pi}_1 \to 
\begin{bmatrix} I_{j_1} & 0 \\ 0 & 0 \end{bmatrix}$. Since $B^* B \to I$,
for any vector $x$ such that $\| x \| = 1$, we have 
$x^* \Pi_1  x - x^* B_1 B_1^* x \to 0$, and 
$x^* \Pi_{B_1} x - x^* B_1 B_1^* x \to 0$. Therefore, 
$x^* (\Pi_1 - \Pi_{B_1}) x \to 0$, which proves the last result. 
% On the other hand, 
% $B_1^* W \to [ I_{j_1}, 0, \ldots, 0]$. Hence $B_1^* \Pi_1 B_1 \to I_{j_1}$
% and similarly for $i > 1$. 

\bibliographystyle{plain} 
\bibliography{math}

\end{document}